\documentclass[11pt,leqno]{amsart}
\usepackage{amsmath,amsfonts,amsthm,amscd,amssymb,graphicx}
\usepackage{epstopdf}
\numberwithin{equation}{section}
\usepackage{fullpage}
\usepackage{color}
\usepackage{hyperref}

\newcommand{\Sp}{\operatorname{Sp}}

\renewcommand\a{\alpha}
\renewcommand\b{\beta}

\def\l{\lambda}
\def\eps{\varepsilon }

\renewcommand\a{\alpha}

\renewcommand\b{\beta}

\newcommand\R{\mathbb R}

\def\eps{\varepsilon}

\def\l{\lambda}

\newcommand\br{\begin{remark}}
\newcommand\er{\end{remark}}
\newcommand\brs{\begin{remarks}}
\newcommand\ers{\end{remarks}}
\newcommand\bp{\begin{pmatrix}}
\newcommand\ep{\end{pmatrix}}
\newcommand{\be}{\begin{equation}}
\newcommand{\ee}{\end{equation}}
\newcommand\ba{\begin{equation}\begin{aligned}}
\newcommand\ea{\end{aligned}\end{equation}}
\newcommand\ds{\displaystyle}

\newcommand{\bap}{\begin{app}}
\newcommand{\eap}{\end{app}}
\newcommand{\begs}{\begin{exams}}
\newcommand{\eegs}{\end{exams}}
\newcommand{\beg}{\begin{example}}
\newcommand{\eeg}{\end{exaplem}}
\newcommand{\bpr}{\begin{proposition}}
\newcommand{\epr}{\end{proposition}}
\newcommand{\bt}{\begin{theorem}}
\newcommand{\et}{\end{theorem}}
\newcommand{\bc}{\begin{corollary}}
\newcommand{\ec}{\end{corollary}}
\newcommand{\bl}{\begin{lemma}}
\newcommand{\el}{\end{lemma}}
\newcommand{\bd}{\begin{definition}}
\newcommand{\ed}{\end{definition}}

\newcommand{\B }{\mathcal{B}}

\newcommand{\RR}{{\mathbb R}}

\newcommand{\ZZ}{{\mathbb Z}}

\newtheorem{theorem}{Theorem}[section]
\newtheorem{proposition}[theorem]{Proposition}
\newtheorem{corollary}[theorem]{Corollary}
\newtheorem{lemma}[theorem]{Lemma}

\theoremstyle{remark}
\newtheorem{remark}[theorem]{Remark}
\newtheorem{remarks}[theorem]{Remarks}
\theoremstyle{definition}
\newtheorem{definition}[theorem]{Definition}

\newtheorem{example}[theorem]{Example}

\newcommand{\lb}{\label}

\newcommand{\ran}{\text{\rm{ran}}}

\newcommand{\dom}{\text{\rm{dom}}}

\newcommand{\beq}{\begin{equation}}
\newcommand{\eeq}{\end{equation}}

\title{
Diffusive stability of spatially periodic patterns with a conservation law
}

\author{Alim Sukhtayev}
\address{Indiana University, Bloomington}
\email{alimsukh@iu.edu}

\begin{document}

\begin{abstract}
Applying the Lyapunov--Schmidt reduction approach introduced by Mielke and Schneider in their analysis of the
fourth-order scalar Swift--Hohenberg equation, we carry out a rigorous small-amplitude stability analysis of Turing
patterns for the model introduced by Matthews and Cox for pattern formation with a conservation law.
Our results confirm that stability is accurately predicted in the small-amplitude limit by the 
formal modified Ginzburg--Lanadau system (mGL) consisting of a coupled Ginzburg--Landau equation and mean mode equation derived by Matthews and Cox, 
rigorously validating the standard weakly unstable approximation.
\end{abstract}

\date{\today}
\maketitle

\tableofcontents

\section{Introduction}

The topic of pattern formation has a wide range of applications and has attracted a lot of interest since the fundamental
observation of Turing \cite{T,C} that reaction diffusion systems 
modeling biological/chemical processes
can develop patterns through destabilization of the homogeneous state.

Besides the question of existence, one of the fundamental topics is stability of periodic patterns and their behavior under small perturbations \cite{E,NW,M1,M2,M3,S1,S2,DSSS,SSSU,JZ,JNRZ1,JNRZ2}.  The formal small-amplitude theory of Eckhaus \cite{E} 
deriving the Ginzburg--Landau equation as a canonical model for behavior near the threshold of instability in
a variety of processes states that a regular periodic pattern is stable provided its wavenumber lies
within the Eckhaus band. The rigorous characterization of spectral stability has been carried out in all details only for
the particular case of the (scalar) Swift-Hohenberg equation \cite{M1,M2,S1} and recently for the Brusselator model \cite{SZJV16}.

However, there is a large class of problems for which the equation governing the modulation of small-amplitude patterns is not the Ginzburg--Landau equation, but the modified Ginzburg--Lanadau system (mGL), that is, \be 
\begin{split}
	\partial_{\hat t} A =& \partial^2_{\hat x} A+A  -|A|^2A-AB,\\
	\partial_{\hat t} B=&\sigma\partial^2_{\hat x} B+ \mu\partial^2_{\hat x}(|A|^2),\,\,\sigma,\mu-\hbox{const.}
\end{split}
\ee Such
situations occur when the system possesses a conservation law \cite{CH93,MC00}. In particular, in \cite{MC00}, a few physical examples in which the 
formal modified Ginzburg--Lanadau system (mGL) arise are described, including convection with fixed-flux boundaries and
convection in a magnetic field.

In this paper we carry out a rigorous small-amplitude stability analysis of Turing
patterns for pattern formation with a conservation law.
Our results confirm that stability is accurately predicted in the small-amplitude limit by the 
formal modified Ginzburg--Lanadau system (mGL) derived by Matthews and Cox \cite{MC00}, 
rigorously validating the standard weakly unstable approximation. To be more specific, our main focus is to consider the following model for pattern formation with a conservation law

\ba\label{sh1c}
\partial_t u & =-\partial_x^2\big[ -(1+\partial_x^2)^2u+\eps^2u-su^2-u^3  \big].
\ea
For model \eqref{sh1c}, there is a Turing instability of the equilibrium state at $\eps=0$, with linear oscillating
modes $c e^{\pm i x}$.
Thus, following standard convention, 
we expect a smooth branch of solutions 
\be\label{branch}
u=\{ \alpha e^{i(1+ \eps \omega ) x}\eps  + O(\eps^2)\} + c.c.,
\ee
bifurcating from $\eps=0$, where $c.c.$ denotes complex conjugate, and $\omega$ lies in an appropriate
range. Following the Lyapunov--Schmidt reduction program laid out in \cite{M1,M2,S1}, we describe the unique branch of solutions bifurcating from equilibrium in a neighborhood of the
Turing instability, and give a detailed description of the spectrum of the linearized operator about the bifurcating solution,
showing that it agrees in the Ginzbur--Landau regime to lowest order with that of the linearization of the modified Ginzburg--Landau system about
the solution corresponding to \eqref{branch}.

The analysis for the Lyapunov--Schmidt reduction program involves a few key factors.
In particular, one needs to use the symmetries of the original problem to obtain the sufficient estimates for remainder terms in the ultimately resulting $3\times 3$ reduced equations to obtain the desired spectral description. Another key factor is to use the Weierstrass preparation theorem which narrows the problem of finding the eigenvalues of the $3\times 3$ reduced spectral matrix to finding the roots of the third degree polynomial. Then one needs to regroup the coefficients of the resulting polynomial in order to decompose it into the first and second degree polynomials using Cardano's formulas. It turns out that the root of the first degree polynomial represents a non-critical eigenvalue of the $3\times 3$ reduced spectral matrix.

Another contribution is to reframe the stability analysis of the modified Ginzburg--Landau system in a way
illuminating the connection with Lyapunov--Schmidt reduction, for which, under appropriate interpretation/scaling,
the two processes can be seen not only to generate the same final results but to match operation-by-operation.
\subsection{Main results}\label{s:mainresults} 
We now state our main results.
Let $H^6_{per}([0,2\pi], \RR)$ denote the space of $H^6$ functions that are periodic on the interval $[0,2\pi]$. Introducing the wave number $k$ and making the independent coordinate change $x\to kx$, we may further normalize the set of periodic solutions
with wave number $k$ to periodic solutions on the fixed interval $[0,2\pi]$ of
\be \label{B model1}
0=N(\eps, k, \tilde u):=-k^2\partial_\xi^2\big[ -(1+k^2\partial_\xi^2)^2\tilde u+\eps^2\tilde u-s\tilde u^2-\tilde u^3  \big],
\qquad N(\eps,k,0)\equiv 0.
\ee
Also, we consider the following problem
\be \label{B model2}
-k^2\big[ -(1+k^2\partial_\xi^2)^2\tilde u+\eps^2\tilde u-s\tilde u^2-\tilde u^3  \big]=q,\,q\hbox{-const.},\,\,u\in H^6_{per}([0,2\pi], \RR).
\ee
Set now $k=\sqrt{1+2\omega\eps}=1+\omega\eps+\mathcal{O}(\eps^2)$.
Our first result rigorously characterizes bifurcation of periodic solutions of \eqref{sh1c} from 
equilibrium state $0$. 

\begin{theorem}[Existence]\label{existence}
	Let $s\in (-\sqrt{27/2},\sqrt{27/2})$. Then there exists an $\eps_0$ such that for all $\eps\in(0,\eps_0)$ and all 
		$\omega\in I_E= [-\frac{1}{2},\frac{1}{2}]$ there is a unique small solution $\tilde u_{\eps, \omega, s}\in H^6_{per}([0,2\pi], \RR)$ of \eqref{B model2} with $q=0$ which is even in $\xi$, positive at $\xi=0$, and has the expansion formula:
	\be
	\begin{split}
	\tilde u_{\eps,\omega,s}(\xi)&=6\sqrt{\frac{1-4\omega^2}{27-2s^2}}\cos \xi \eps+\big(-32\omega s^2\sqrt{\frac{1-4\omega^2}{(27-2s^2)^3}}\cos\xi-\frac{2s(1-4\omega^2)}{27-2s^2}\cos 2\xi\big)\eps^2+\mathcal{O}(\eps^3).
	\end{split}
	\ee
	Note that when $\omega=\pm\frac{1}{2}$, $\tilde u_{\eps,\omega,s}\equiv0$ reduces to the equilibrium (zero) solution.
\end{theorem}

\begin{proof}
Given in Section \ref{s:existence}.
\end{proof}

Our second result rigorously characterizes diffusive stability/instability of bifurcating solutions.

Linearizing \eqref{sh1c} about $\tilde u_{\eps, \omega, s}$, we have
\be 
\begin{split}
	\hat B_{\eps, \omega, s}(\partial_\xi)&:\dom(\hat B_{\eps, \omega, s}(\partial_\xi))=H^6(\mathbb{R})\subset L^2(\mathbb{R})\to L^2(\mathbb{R}),\\
\hat B_{\eps, \omega, s}(\partial_\xi)v&:=-k^2\partial_\xi^2\big[ -(1+k^2\partial_\xi^2)^2+df(\tilde u_{\eps, \omega, s}) \big]v.
\end{split}
\ee
Next, we define the Bloch operator family: for $\sigma \in \RR$,
\be
\begin{split}
B(\eps, \omega, s,\sigma) &: L^2_{per} ([0,2\pi])\subset H^6_{per}([0,2\pi]) \longrightarrow  L^2_{per} ([0,2\pi]),\\
B(\eps, \omega, s, \sigma)V &:= -k^2(\partial_\xi + i\sigma)^2\big[ -(1+k^2(\partial_\xi + i\sigma)^2)^2+df(\tilde u_{\eps, \omega, s}) \big]V,
\end{split}
\ee
where $k=\sqrt{1+2\omega\eps}=1+\omega\eps+\mathcal{O}(\eps^2)$.

\begin{theorem}[Stability]\label{stabthm}
	Let $u_{\eps,\omega,s}$ be the solution from Theorem \ref{existence}.  Then there exist
	 $\tilde\eps_0\in(0,\eps_0]$, where $\eps_0$ is taken from Theorem \ref{existence}, $\sigma_0>0$ and $\delta>0$ such that for all $\eps\in(0,\tilde\eps_0)$, all $\sigma\in[0,\sigma_0)$, all $\omega\in[-\frac{1}{2},\frac{1}{2}]$ and all $s\in (-\sqrt{27/2},\sqrt{27/2})$ the spectrum of $B(\eps, \omega, s,\sigma)$ has the decomposition:
	\be
	\begin{split}
		\Sp(B(\eps, \omega, s,\sigma))=S\cup\{\lambda_1,\lambda_2,\lambda_3\},
	\end{split}
	\ee
	where $\Re\lambda<-\delta$ for $\lambda\in S$ and $|\lambda_j|<<1$.
	 Moreover, for $s\in[-\sqrt{\frac{27}{38}}, \sqrt{\frac{27}{38}}]$  and each fixed 
	 $\omega\in I_S= (-\min\{\frac{1}{2},\sqrt{\frac{27-38s^2}{12(27-14s^2)}}\},\min\{\frac{1}{2},\sqrt{\frac{27-38s^2}{12(27-14s^2)}}\})$ there exists $\hat\eps_0\in(0,\tilde\eps_0)$ such that for all $\eps\in(0,\hat\eps_0)$, all $\sigma\in[0,\sigma_0)$
	 \be\label{lamexp}
	\begin{split}
		&	\Re\lambda_{1}\leq c(\eps, \omega)+\tilde c(\eps, \omega)\sigma-\tilde{\tilde c}_1(\eps, \omega, s)\sigma^2+\mathcal O(|\sigma|^3),\\
		&	\Re\lambda_{2}\leq -\tilde{\tilde c}_2(\eps, \omega, s)\sigma^2+\mathcal O(|\sigma|^3),\\
		&	\Re\lambda_{3}\leq -\tilde{\tilde c}_3(\eps, \omega, s)\sigma^2+\mathcal O(|\sigma|^3),
	\end{split}
	\ee
	for $c(\eps,\omega)<0<\tilde{\tilde c}_j(\eps, \omega, s)$, giving \emph{diffusive stability},
	while if $s\in(\sqrt{\frac{27}{2}},-\sqrt{\frac{27}{38}}) \cup(\sqrt{\frac{27}{2}},\sqrt{\frac{27}{38}})$
	 or $\omega \in I_E\setminus \overline{I_S}$, then
	\be\label{unst}
\max_\sigma\{\Re\lambda_1,\Re\lambda_2,\Re\lambda_3\}>0,
\ee
giving \emph{diffusive instability}.

\end{theorem}

\begin{proof}
Given in Section \ref{s:stability}.
\end{proof}

Theorems \ref{existence} and \ref{stabthm} together rigorously validate the predictions of the modified Ginzburg--Landau approximation
regarding existence and stability of small bifurcating solutions.

Our third main result states that, within the Ginzburg--Landau regime $\lambda\sim \eps^2$, $\sigma \sim \eps$,
the Ginzburg Landau approximation not only well-predicts stability/instability, but to lowest order also
the linearized dispersion relations for the two smallest eigenmodes.

\bt\label{dispthm}
Setting $\sigma =:\eps \hat \sigma$, $\lambda_j=:\eps^2 \hat \lambda_j$ in accordance with the Ginzburg--Landau 
scaling, $\lambda_j$ as in \eqref{lamexp}, we obtain expansions 
\be\label{exactdisp}
	 		 		 	\begin{split}
	&\hat\lambda_{1}(\hat\sigma)=-2(1-4\omega^2)\eps^2+(-\frac{36s^2}{27-2s^2}-\frac{4(1+4\omega^2)}{1-4\omega^2})\hat\sigma^2+\mathcal O(|\hat\sigma|^3 + \eps(1+|\hat\sigma|)),\\
	 		 		 	&\hat\lambda_{2}(
	 		 		\hat\sigma)=\lambda_-\hat\sigma^2
		+\mathcal O(|\hat\sigma|^3 +\eps \hat \sigma^2),\\
		&\hat\lambda_{3}(
		\hat\sigma)=\lambda_+\hat\sigma^2
		+\mathcal O(|\hat\sigma|^3 +\eps \hat \sigma^2),\\
			&\lambda_{\pm}=\frac{(-5+\frac{36s^2}{27-2s^2}+\frac{32\omega^2}{1-4\omega^2})\pm\sqrt{(5-\frac{36s^2}{27-2s^2}-\frac{32\omega^2}{1-4\omega^2})^2-4(4-\frac{144s^2}{27-2s^2}-\frac{32\omega^2}{1-4\omega^2})}}{2}.
	 		 		 		\end{split}
	 		 		 			\ee
for $|\hat\sigma|<<1$,
agreeing to lowest order with the corresponding expansions for the associated modified Ginzburg--Landau
approximation (cf. \eqref{gldisp}).
\et

\begin{proof}
Given in Section \ref{s:GLcompare}.
\end{proof}

\subsection{Discussion and open problems}\label{s:discussion}

In \cite{BJK16}, the authors give several examples about
bifurcating stable periodic waves from uniform states by Turing instability for parabolic systems of conservation laws. That is 
\be
\label{fullcl}
u_t + f(\eps, u)_x=(D(u)u_x)_x,
\ee
$u\in \R^n$.
These results suggest an open problem of deriving the amplitude equations governing the modulation of the small-amplitude patterns for \eqref{fullcl}, and rigorously validating the standard weakly unstable approximation.

\section{Existence of periodic solutions}\label{s:existence}
In this section we study existence of periodic solutions, carrying out the proof of Theorem \ref{existence}.

Then  instability occurs at $\eps=0$ with the corresponding wave number $k=1$ and we consider a two-paramametric $(\eps, k)$ family of stationary solutions $\tilde u_{\eps,k}$ which bifurcate for $\eps=0$ from $u^*=0$. In order to show that there are bifurcating periodic stationary solutions from $u^*=0$, we use the Lyapunov-Schmidt reduction.  Let $\tilde u(\eps,k,\xi)$ be $2\pi$ - periodic in $\xi$ where $\xi=kx$ (that is, we assume $\frac{2\pi}{k}$ - periodic in $x$). We will look at the expression of the periodic solution $\tilde u$ in a neighborhood of $(\eps,k,u)=(0, 1, 0)$. By \eqref{sh1c},  $\tilde u$ satisfies
\be \label{B model}
0=N(\eps, k, \tilde u):=-k^2\partial_\xi^2\big[ -(1+k^2\partial_\xi^2)^2 \tilde u+\eps^2 \tilde u-s \tilde u^2- \tilde u^3  \big],
\ee
where  $N: \RR^2 \times H^6_{per}([0,2\pi], \RR) \longrightarrow  L^2_{per}([0,2\pi], \RR) $ is a $C^{\omega}$ mapping.

\subsection{The Lyapunov-Schmidt reduction for the equation \eqref{B model}}

\noindent We first sketch the Lyapunov-Schmidt reduction for the equation \eqref{B model}. Since $N(0, 1, 0) =0$, we want to study the stationary periodic solutions of the equation \eqref{B model} in a neighborhood of $(0, 1, 0)$ in $\RR^2 \times H^6_{per}([0,2\pi], \RR)$. We define
\be
L_{per}:=\partial_{\tilde u}N(0, 1, 0) = \partial_\xi^2(1+\partial_\xi^2)^2.
\ee

If $L_{per}$ was invertible and $L^{-1}_{per}$ was bounded from $L^2_{per}([0,2\pi], \RR)$ to $H^6_{per}([0,2\pi], \RR)$, then by the Implicit Function Theorem, in a neighborhood of $(0, 1, 0)$, there would exist a unique solution $\tilde u(\xi) = \phi (\eps,k)$ satisfying \eqref{B model} for some $C^{\omega}$ function $\phi$. In this case, however, $L_{per}$ is not invertible, so we apply the Lyapunov-Schmidt reduction. We first denote the kernel and range of $L_{per}$ by $\ker (L_{per})$ and $\ran (L_{per})$, respectively. Moreover, we assume the decompositions:
\be
H^6_{per}([0,2\pi], \RR)=\ker(L_{per}) \oplus X_1 \quad \text{and}  \quad L^2_{per}([0,2\pi], \RR) = \ran(L_{per}) \oplus Y_1,
\ee
where $X_1$ and $Y_1$ are topological complements of $\ker(L_{per})$ and $\ran (L_{per})$ in $H^6_{per}([0,2\pi], \RR)$ and $L^2_{per}([0,2\pi], \RR)$, respectively. Then there are two continuous projections $P:H^6_{per}([0,2\pi], \RR) \longrightarrow H^6_{per}([0,2\pi], \RR)$ and $Q : L^2_{per}([0,2\pi], \RR)  \longrightarrow L^2_{per}([0,2\pi], \RR)$ such that
\be \label{condition of projections}
\ran (P)=\ker (L_{per}) \quad \text{and} \quad \ker(Q)=\ran(L_{per}),
\ee
that is, $P(H^6_{per}([0,2\pi], \RR))=\ker(L_{per})$ and $(I-Q)(L^2_{per}([0,2\pi], \RR))=\ran(L_{per})$. Now we decompose $\tilde u \in H^6_{per}([0,2\pi], \RR)$ into $U + V$, where $U=P\tilde u \in \ker (L_{per})$ and $V=(I-P)\tilde u  \in X_1$. Then one can rewrite the equation \eqref{B model} as
\be\label{two equations}
QN(\eps, k,  U +V) = 0, \quad (I-Q)N(\eps, k,  U +V) = 0.
\ee
We first focus on the second equation. Defining
\be
G(\eps, k, U, V) := (I-Q)N(\eps, k, U +V) = 0,
\ee
notice that $\ds G(0, 1, 0)=(I-Q)N(0, 1, 0)=0$ and
\be
\partial_V G(0, \pm 1,0)=(I-Q)\partial_{\tilde u}N(0, 1, 0)= (I-Q)L_{per}=L_{per}.
\ee
Since $L_{per}:(I-P)H^6_{per}([0,2\pi], \RR)  \longrightarrow (I-Q)L^2_{per}([0,2\pi], \RR)$ is bijective, applying the Implicit Function Theorem, $G(\eps,k,U,V)$ can be solved for $V$ in $(I-P)H^6_{per}([0,2\pi], \RR)$ as a function of $(\eps,k,U)$. More precisely, there exist an open neighborhood $\Omega$ of $U=0$ in $\ker(L_{per})$, an open neighborhood $\Gamma$ of $(\eps,k)=(0, 1)$ in $\RR^2$, and a $C^{\omega}$ function $\Phi : \Omega \times \Gamma \longrightarrow \ker(P) (=X_1) $ such that $\Phi (0, 1, 0) = 0$ and
\be
(I-Q)N(\eps,k,U+\Phi (\eps, k, U) ) \equiv 0,
\ee
for all $(\eps,k,U) \in \Gamma \times \Omega$. We now substitute $V=\Phi(\eps,k,U)$ into the first equation of \eqref{two equations} in order to obtain the bifurcation equation:
\be \label{GFE}
QN(\eps,k, U +\Phi (\eps, k, U) )= 0.
\ee
Setting
\be
B(\eps,k,U)=QN(\eps,k, U +\Phi (\eps, k, U) ),
\ee
B is a $C^{\omega}$ function from $\Gamma \times \Omega$ to $Y_1$ which has a finite dimension, $B(0, 1,0)=0$ and $\partial_U B(0, 1, 0)=0$. Actually, solving \eqref{GFE} is equivalent to solving the original equation \eqref{B model}, that is, it is enough to solve the finite -dimensional problem $B(\eps,k,U)=0$ locally in $\RR^2 \times \ker( L_{per})$.

\begin{remark}
In the above argument, $(I-P)H^6_{per}([0,2\pi], \RR) =  \ker(P) = X_1$
\end{remark}

\subsection{Periodic solutions $\tilde u$ of \eqref{B model}}

By linearization of \eqref{B model} about $u^* =0$, we have
\be \label{linearization}
\partial_{\tilde u}N(\eps, k, u^*)[U]=-k^2\partial_\xi^2\big( -(1+k^2\partial_\xi^2)^2+\eps^2\big)[U].
\ee

In particular, putting $\eps=0$, $k^2=1$, we have
\be \label{linear}
L_{per}U:=\partial_{\tilde u}N\left(0,1, 0\right)[U]= \partial_\xi^2(1+\partial_\xi^2)^2[U].
\ee
Note that $L_{per}$ is self-adjoint and the kernel of $L_{per}$ is spanned by
\be \label{kernel}
U_1(\xi)=\cos \xi \quad  \quad U_2(\xi)=\sin \xi \quad  \text{and} \quad U_3(\xi)=1.
\ee

Now, in order to use Lyapunov-Schmidt reduction, we first define the zero eigenprojection
\be \label{projection P}
Qu= \langle  U_1, u\rangle U_1+\langle  U_2, u\rangle U_2+\frac{1}{2}\langle  U_3, u\rangle U_3, \quad \text{where} \quad \langle u, v \rangle = \frac{1}{\pi} \int_0^{2\pi}u \cdot v d\xi,
\ee
and define the mapping
\be \label{projection tilde P}
\tilde  Q: L^2_{per}([0,2\pi], \RR) \rightarrow \RR^3 ; u \mapsto (\langle U_1, u\rangle, \langle U_2, u\rangle, \frac{1}{2}\langle U_3, u\rangle)^{T},
\ee
that is, $\tilde Q$ is just a vector form in $\RR^3$ of the projection $Q$.  
Decomposing $\tilde u  \in H^6_{per}([0,2\pi], \RR)$ into $\a_1 U_1+\a_2 U_2+\a_3 U_3  + V$, where $PV=0$, we see that the linearization \eqref{linear} is invertible on $(I-P)H^6_{per}([0,2\pi], \RR)$. Moreover, recalling \eqref{B model}, we have
\be \label{projection equations}
\begin{split}
& \tilde  QN(\eps,k, \a_1 U_1+\a_2 U_2 + V )=0, \\
& (I-Q)N(\eps,k, \a_1 U_1+\a_2 U_2+\a_3 U_3 + V)=0,
\end{split}
\ee
where
\be
(I-Q)N(\eps,k, \a_1 U_1+\a_2 U_2 +\a_3 U_3+ V):\mathbb{R}^5\times\ran(I-P)\to\ran(I-Q).
\ee

By the Implicit Function Theorem, there exists  an open neighborhood $U\subset\mathbb{R}^5$ of $(0,1,0,0,0)$ and a unique function $V: U\to(I-P)H^6_{per}([0,2\pi], \RR)$ that solves the second equation of \eqref{projection equations} for $(\eps,k,\alpha_1,\alpha_2,\a_3)\in U$. After we substitute $V$ into the first equation of \eqref{projection equations}, the reduced system (or the bifurcation system) will be $O(2)$ equivariant with respect to $\a_1$ and $\a_2$. This is due to the fact that the original problem is translation invariant and reflection symmetric. And since there is second derivative with respect to $\xi$ in \eqref{B model} the projection onto $U_3$ in \eqref{projection tilde P} is $0$. Hence, we can conclude that the nontrivial part of the reduced system is of the form

\be\label{be}
 f(\eps,k,\a_1^2+\a_2^2,\a_3)\bp \a_1 \\ \a_2 \ep=0,
 \ee

  $f$ is a real-valued scalar function (c.f. \cite[Chapters 2,5]{CL}).

Next, let us find asymptotic expansion of $V$ with respect to parameter $\alpha_1$ and set $\alpha_2=0$.\\
1.) First of all, it is clear that $V(\eps,k,0,0,\alpha_3)=0$.\\
2.) Now, we differentiate the second equation of \eqref{projection equations} with respect to $\alpha_1$.
\be\label{alpha1}
\begin{split}
& \partial_{\alpha_1}(I-Q)N(\eps,k, \a_1 U_1+\a_3 U_3 + V)=(I-Q)[-k^2\partial_\xi^2\big([ -(1+k^2\partial_\xi^2)^2+\epsilon^2] (U_1+\partial_{\alpha_1}V)\\
&-(2s u+ 3u^2)(U_1+\partial_{\alpha_1}V)  \big)]=0,
\end{split}
\ee
where $u=\a_1 U_1+\a_3 U_3 + V$.\\
Hence, by step 1.),
\be\label{partial1}
\begin{split}
& \partial_{\alpha_1}(I-Q)N(\eps,k, \a_1 U_1+\a_3 U_3 + V)|_{\alpha_1=0}\\
&=(I-Q)[-k^2\partial_\xi^2[ -(1+k^2\partial_\xi^2)^2+\epsilon^2-2s\a_3-3\a^2_3] (\partial_{\alpha_1}V)]=0.
\end{split}
\ee

Since the operator $(I-Q)[-k^2\partial_\xi^2( -(1+k^2\partial_\xi^2)^2+\epsilon^2-2s\a_3-3\a^2_3)](I-P)$ is invertible in the neighborhood of $(0,1,0,0,0)$, $\partial_{\alpha_1}V(\eps,k,\a_1,0,\a_3)|_{\alpha_1=0}=0$. \\
3a.) Next, we would like to compute $\partial^2_{\alpha_1}V(\eps,k,\a_1,0,\a_3)|_{\alpha_1=0}$.\\
We differentiate \eqref{alpha1} with respect to $\alpha_1$.\\
\be\label{alpha1^2}
\begin{split}
& \partial^2_{\alpha_1}(I-Q)N(\eps,k, \a_1 U_1+\a_3 U_3 + V)=(I-Q)[-k^2\partial_\xi^2\big([ -(1+k^2\partial_\xi^2)^2+\epsilon^2] (\partial^2_{\alpha_1}V)\\
&-(2s+ 6u)(U_1+\partial_{\alpha_1}V)^2 -(2s u+ 3u^2)\partial^2_{\alpha_1}V \big)]=0.
\end{split}
\ee
Therefore,
\be\label{partial1^2}
\begin{split}
& \partial^2_{\alpha_1}(I-Q)N(\eps,k, \a_1 U_1+\a_3 U_3 + V)|_{\alpha_1=0}\\
&=(I-Q)[-k^2\partial_\xi^2\big([ -(1+k^2\partial_\xi^2)^2+\epsilon^2-(2s\a_3+3\a_3^2)] (\partial^2_{\alpha_1}V)-(2s+6\a_3)U_1^2 \big)]=0.
\end{split}
\ee

Therefore, using \eqref{partial1^2}, we obtain
\be
\begin{split}\label{formula1^2a}
&(I-Q)[-k^2\partial_\xi^2\big([ -(1+k^2\partial_\xi^2)^2+\epsilon^2-(2s\a_3+3\a_3^2)] (\partial^2_{\alpha_1}V)|_{\alpha=0}\big)=4(s+3\a_3)k^2 \cos2 \xi.
\end{split}
\ee
Since $span\{\cos 2\xi\}$ is an invariant subspace for the invertible operator $(I-Q)[-k^2\partial_\xi^2( -(1+k^2\partial_\xi^2)^2+\epsilon^2-(2s\a_3+3\a_3^2))](I-P)$, $\partial^2_{\alpha_1}V|_{\alpha_1=0}$ is of the form $a \cos 2\xi$. It follows from \eqref{formula1^2a} that
\be
\begin{split}
& 4k^2( -(1-4k^2)^2+\epsilon^2-(2s\a_3+3\a_3^2))a=4(s+3\a_3)k^2.
\end{split}
\ee
Therefore,
\be\label{partiala11}
\begin{split}
 &\partial^2_{\alpha_1}V|_{\alpha=0}=\frac{s+3\a_3}{ -(1-4k^2)^2+\epsilon^2-(2s\a_3+3\a_3^2)}\cos 2\xi.
 \end{split}
 \ee
4.) Next, we compute  $\partial^3_{\alpha_1}V(\eps,k,\a_1,0,\a_3)|_{\alpha_1=0}$.\\
We differentiate \eqref{alpha1^2} with respect to $\alpha_1$.\\
\be\label{alpha1^3}
\begin{split}
\begin{split}
& \partial^3_{\alpha_1}(I-Q)N(\eps,k, \a_1 U_1+\a_3 U_3 + V)=(I-Q)[-k^2\partial_\xi^2\big([ -(1+k^2\partial_\xi^2)^2+\epsilon^2] (\partial^3_{\alpha_1}V)\\
&-6(U_1+\partial_{\alpha_1}V)^3-3(2s+ 6u)(U_1+\partial_{\alpha_1}V)\partial^2_{\alpha_1}V -(2s u+ 3u^2)\partial^3_{\alpha_1}V \big)]=0.
\end{split}
\end{split}
\ee

Therefore,
\be\label{partial1^3}
\begin{split}
& \partial^3_{\alpha_1}(I-Q)N(\eps,k, \a_1 U_1+\a_3 U_3 + V)|_{\alpha_1=0}=(I-Q)[-k^2\partial_\xi^2\big([ -(1+k^2\partial_\xi^2)^2\\
&\epsilon^2-(2s \a_3+ 3\a_3^2)] (\partial^3_{\alpha_1}V)-6U_1^3-3(2s+6\a_3)U_1\partial^2_{\alpha_1}V \big)].
\end{split}
\ee

It follows from \eqref{partiala11}  that
\be\label{partial1^3ccc}
\begin{split}
(I-Q)[-k^2\partial_\xi^2\big( -(1+k^2\partial_\xi^2)^2+\epsilon^2-(2s \a_3+ 3\a_3^2))] (\partial^3_{\alpha_1}V)=(\ldots)\cos 3\xi.
\end{split}
\ee

Since $span\{\cos 3\xi\}$ is an invariant subspace for the invertible operator $(I-Q)[-k^2\partial_\xi^2( -(1+k^2\partial_\xi^2)^2+\epsilon^2-(2s \a_3+ 3\a_3^2))](I-P)$ $\partial^3_{\alpha_1}V|_{\alpha=0}$ is of the form $b(\a_3)\cos 3\xi$.

So far, we have shown that 
 \be
 \begin{split}\label{V}
 &V(b,k,\a_1,0,\a_3)=\frac{1}{2}\big(\frac{s+3\a_3}{ -(1-4k^2)^2+\epsilon^2-(2s\a_3+3\a_3^2)}\cos 2\xi\big)\alpha_1^2\\
 &+\frac{1}{6}(b(\a_3)\cos 3\xi)\alpha_1^3+\mathcal{O}(|\a_1|^4).
 \end{split}
 \ee

Let $f$ denote the coefficient in front of $\alpha_1^2$ in \eqref{V}, that is,
\be
\begin{split}
	&f=\frac{s+3\a_3}{2( -(1-4k^2)^2+\epsilon^2-(2s\a_3+3\a_3^2))}.
\end{split}
\ee

In order to obtain the reduced system, we substitute \eqref{V} into the first equation from \eqref{projection equations}, obtaining for $u=\a_1 U_1 + \a_3 U_3+ V$ the equation

\be\label{QN}
 \tilde QN(b,k, \a_1 U_1 + \a_3 U_3+ V) =\tilde Q \Big[-k^2\partial_\xi^2\big[ \big(-(1+k^2\partial_\xi^2)^2+\eps^2\big) (\a_1 U_1 + \a_3 U_3+ V)-s u^2- u^3  \big] \Big]=0.
\ee

Next, we split the left-hand side of \eqref{QN} into two parts.

a.) {\bf Linear part.}

We have the following expression for the linear part from \eqref{QN}
\be
\begin{split}
& \tilde Q \Big[-k^2\partial_\xi^2\big[\big(-(1+k^2\partial_\xi^2)^2+\eps^2\big) (\a_1 U_1 + \a_3 U_3+ V)\big] \Big]=(-(1-k^2)^2+\eps^2)k^2\a_1  +\mathcal{O}(|\a_1|^4)\bp 1 \\ 0\\0 \ep.
\end{split}
\ee

b.) {\bf Non-linear part.}
We next treat the nonlinear terms
\be\label{nonlinear}
\begin{split}
&\tilde Q \Big[-k^2\partial_\xi^2\big[-s (\a_1 U_1 + \a_3 U_3+ V)^2- (\a_1 U_1 + \a_3 U_3+ V)^3    \big] \Big]=[(-2s-3\a_3)k^2\a_3\a_1\\
&+(-sf-\frac{3}{4}-3\a_3f)k^2\a_1^3+\mathcal{O}(|\a_1|^4)]\bp 1 \\ 0\\0 \ep.
\end{split}
\ee    
  
Hence, taking into account formulas \eqref{be} and \eqref{QN}, the nontrivial part of the reduced system has the form:
\be\label{reduced}
\begin{split}
&\big\{ [(-(1-k^2)^2+\eps^2-2s\a_3-3\a_3^2)] k^2  - 
(\frac{3}{4}+sf+3\a_3f)k^2(\a_1^2+\a_2^2)  +\mathcal{O}(|\a|^4)\big\}\bp \a_1 \\ \a_2 \ep=0.
\end{split}
\ee
Let us introduce $\mathcal{A}$:
\be
\begin{split}
	\mathcal{A}=&\frac{-(1-k^2)^2+\eps^2-2s\a_3-3\a_3^2}{\frac{3}{4}+sf+3\a_3f}  ,
\end{split}
\ee
Following [MC00] we take  $\a_2=0$, $\a_3=0$ and $\a_1=\a$. Then

\be
\begin{split}
	\mathcal{A}=&\frac{-(1-k^2)^2+\eps^2}{\frac{3}{4}+sf}  ,
\end{split}
\ee
where 
\be
\begin{split}
	&f=\frac{s}{2( -(1-4k^2)^2+\epsilon^2)}.
\end{split}
\ee
Or,
\be\label{Aasymp}
\begin{split}
	\mathcal{A}=&\frac{4( -(1-4k^2)^2+\epsilon^2)}{3(-(1-4k^2)^2+\epsilon^2)+2s^2} (\eps^2-(1-k^2)^2)=\big(\frac{36}{27-2s^2}+\mathcal O(k-1,\eps)\big)(\eps^2-(1-k^2)^2).
\end{split}
\ee
Fron now on, we will assume that $27-2s^2>0$.\\
Our goal is to solve \eqref{reduced} for $\a$ in terms of $\eps$ and $k$. 

Solving \eqref{reduced} is equivalent to solving
\be\label{short}
\begin{split}
&\mathcal{A}-\a^2  +\mathcal{O}(|\a|^4)=0.
\end{split}
\ee
Next, plugging $\a=\sqrt{|\mathcal{A}|}\mathcal{B}$ into \eqref{short}, we obtain
\be
\begin{split}
&\mathcal{A}-|\mathcal{A}|\mathcal{B}^2  +\mathcal{O}(|\mathcal{A}|^2)=0,
\end{split}
\ee
or
\be
\begin{split}\label{B}
&\mathcal{A}(1-\mathcal{B}^2+\mathcal{O}(|\mathcal{A}|))=0\,\,\hbox{if}\,\, \mathcal{A}\geq0,\\
&\mathcal{A}(1+\mathcal{B}^2+\mathcal{O}(|\mathcal{A}|))=0\,\,\hbox{if}\,\, \mathcal{A}\leq0.
\end{split}
\ee
We need to solve \eqref{B} in terms of $\mathcal{A}$. The second equation in \eqref{B} has no solutions. By the Implicit Function Theorem, there exists  an open neighborhood $U\subset\mathbb{R}$ of $0$ and a unique function $\mathcal{B}: U\to\R$ that solves the first equation of \eqref{B} for $\mathcal{A}\in U$. Therefore, we have the restriction $\mathcal{A}\geq0$. Hence, using formula \eqref{Aasymp}, the restrictions on $s$ and $\mathcal{A}$, we conclude that $\eps^2-(1-k^2)^2$ must be greater or equal to $0$. Next, we introduce a scaling parameter $\omega$ defined by the equation
\be
\begin{split}\label{w}
1-k^2=-2\omega\eps.
\end{split}
\ee
Therefore,
\be
\begin{split}\label{k1}
k=\sqrt{1+2\omega\eps}.
\end{split}
\ee
Then, 
\be
\begin{split}
\mathcal{A}\geq0\,\,\hbox{if and only if}\,\, k=\sqrt{1+2\omega\eps}\,\,\hbox{and}\,\,\omega\in[-\frac{1}{2},\frac{1}{2}].
\end{split}
\ee
Note that when $\omega=\pm\frac{1}{2}$, $\mathcal{A}=0$.

Next, using the first equation in \eqref{B}, we arrive at the asymptotic formula for $\mathcal{B}$:
\be
\begin{split} 
\mathcal{B}=1+\mathcal{O}(\mathcal{A}),
\end{split}
\ee
which implies that
\be\label{alphatoo}
\a=\sqrt{|\mathcal{A}|}\mathcal{B}=\sqrt{\mathcal{A}}+\mathcal{O}(\mathcal{A}^{3/2}).
\ee
Since, $k$ is a function of $\eps$. $\mathcal{A}$ is a function of $\eps$ as well. In particular,
\be
\begin{split} \label{Acal}
	\mathcal{A}=\frac{36(1-4\omega^2)}{27-2s^2}\eps^2-\frac{384\omega s^2(1-4\omega^2)}{(27-2s^2)^2}\eps^3+\mathcal{O}(\eps^3).
\end{split}
\ee
Therefore, using \eqref{alphatoo}, we arrive at the asymptotic formula for $\a$:
\be
\begin{split} \label{alpha}
	\a=6\sqrt{\frac{1-4\omega^2}{27-2s^2}}\eps-32\omega s^2\sqrt{\frac{1-4\omega^2}{(27-2s^2)^3}}\eps^2+\mathcal{O}(\eps^2).
\end{split}
\ee
Note that when $\omega=\pm\frac{1}{4}$, $\a=0$.

Using formulas \eqref{V} and \eqref{alpha}, we 
obtain the result of Theorem \ref{existence}.


\section{Stability of periodic solutions }\label{s:stability}
In this section we study stability of the bifurcating periodic solutions established in Section \ref{s:existence}, carrying out the proof of Theorem \ref{stabthm}.
Linearizing \eqref{sh1c} about $\tilde u_{\eps, \omega, s}$, we have
\be 
\hat B_{\eps, \omega, s}(\partial_\xi)v:=-k^2\partial_\xi^2\big[ -(1+k^2\partial_\xi^2)^2+df(\tilde u_{\eps, \omega, s}) \big]v,
\ee
where $\hat B_{\eps, \omega, s}(\partial_\xi):\dom(\hat B_{\eps, \omega, s}(\partial_\xi))=H^6(\mathbb{R})\subset L^2(\mathbb{R})\to L^2(\mathbb{R})$ and 
\be\label{df}
\begin{split}
df(\tilde u_{\eps, \omega, s})
& = \eps^2-2s\tilde u_{\eps, \omega, s}-3\tilde u^2_{\eps, \omega, s}\\
& = -12s\sqrt{\frac{1-4\omega^2}{27-2s^2}}\cos\xi\eps+(1-\frac{54(1-4\omega^2)}{27-2s^2}+64\omega s^3\sqrt{\frac{1-4\omega^2}{(27-2s^2)^3}}\cos\xi\\
&-2(1-4\omega^2)\cos 2\xi)\eps^2+\mathcal O(\eps^3).
\end{split}
\ee
Therefore,
\be \label{linearization about tilde u}
\begin{split}
\hat B_{\eps, \omega, s}(\partial_\xi)v&=-\partial_\xi^2\big[ -(1+\partial_\xi^2)^2+\{-2\omega(1+\partial_\xi^2)^2-4\omega\partial_\xi^2(1+\partial_\xi^2)-12s\sqrt{\frac{1-4\omega^2}{27-2s^2}}\cos\xi\}\eps\\
&+\{-8\omega^2\partial_\xi^2(1+\partial_\xi^2)-4\omega^2\partial_\xi^4+1-\frac{54(1-4\omega^2)}{27-2s^2}+8\omega s(14s^2-81)\sqrt{\frac{1-4\omega^2}{(27-2s^2)^3}}\cos\xi\\
&-2(1-4\omega^2)\cos 2\xi\}\eps^2+\mathcal O(\eps^3)\big]v.
\end{split}
\ee
Since $df(\tilde u_{\eps, \omega, s})$ is $2\pi$-periodic, every coefficient of the linear operator $\hat B_{\eps, \omega, s}$ is $2\pi$-periodic. By substituting $v(\xi)= e^{i \sigma \xi} V(\xi)$ we define the Bloch operator family: for $\sigma \in \RR$,
\be\label{Bloch}
B(\eps, \omega, s, \sigma)V= -k^2(\partial_\xi + i\sigma)^2\big[ -(1+k^2(\partial_\xi + i\sigma)^2)^2+df(\tilde u_{\eps, \omega, s}) \big]V,
\ee
where $B(\eps, \omega, s,\sigma) : L^2_{per} ([0,2\pi])\subset H^6_{per}([0,2\pi]) \longrightarrow  L^2_{per} ([0,2\pi])$ and $k=1+\omega\eps+\mathcal O(\eps^2)$.

However, in order to study the spectral stability of $\tilde u_{\eps, \omega, s}$, it is enough to consider $\sigma \in [-\frac{1}{2},\frac{1}{2})$ because for any $\sigma \in \RR$, $\sigma = \sigma^*+m$, where $\sigma^* \in  [-\frac{1}{2},\frac{1}{2})$ and $m \in \ZZ$; hence we consider $e^{im\xi}V(\xi)$ instead of $V(\xi)$. We now  define the operator $B_0$:
\be
B_0(\sigma):=B(0,\omega,s,\sigma)=-(\partial_\xi + i\sigma)^2\big[ -(1+(\partial_\xi + i\sigma)^2)^2\big],
\ee
which has constant coefficients. Here, we consider Bloch operators $B(\eps, \omega, s,\sigma)$ as small perturbations of $B_0(\sigma)$. So we first study the eigenvalue problem of $B_0(\sigma)$:
\be
B_0(\sigma) e^{im\xi} = \mu_m  e^{im\xi}.
\notag
\ee
It is clear that $\mu_m=(m+\sigma)^2\big[ -(1-(m + \sigma)^2)^2\big]$. Therefore, as long as $\sigma$ is bounded away from $0$, the spectrum of $B_0(\sigma)$ has negative upper bound. Moreover,
$B(\eps, \omega, s, \sigma)-B_0(\sigma)$ represents a small perturbation, that is, $B(\eps, \omega, s, \sigma)-B_0(\sigma)$ is $B_0(\sigma)$-bounded (i.e. $\dom(B_0(\sigma))\subset\dom\big(B(\eps, \omega, s, \sigma)-B_0(\sigma)\big)$, $\|\big(B(\eps, \omega, s, \sigma)-B_0(\sigma)\big)f\|\leq a(\eps,\sigma)\|f\|+b(\eps,\sigma)\|B_0(\sigma)f\|$ for any $f\in\dom(B_0(\sigma))$, and $a(\eps,\sigma),b(\eps,\sigma)\to0$ as $\eps\to0$). Then if $\lambda\in\rho((B_0(\sigma))$, we arrive at
\be
\begin{split}
 B(\eps, \omega, s, \sigma)-\lambda&=B_0(\sigma)-\lambda+B(\eps, \omega, s, \sigma)-B_0(\sigma)\\
 &=[I+(B(\eps, \omega, s, \sigma)-B_0(\sigma))(B_0(\sigma)-\lambda)^{-1}](B_0(\sigma)-\lambda).
\end{split}
\ee
Note that $(B(\eps, \omega, s, \sigma)-B_0(\sigma))(B_0(\sigma)-\lambda)^{-1}\in\mathcal{B}(L^2_{per} ([0,2\pi]))$. And,
\be
\begin{split}
\|(B(\eps, \omega, s, \sigma)-B_0(\sigma))(B_0(\sigma)-\lambda)^{-1}f\|&\leq a(\eps,\sigma)\|(B_0(\sigma)-\lambda)^{-1}f\|\\
&+b(\eps,\sigma)\|B_0(\sigma)(B_0(\sigma)-\lambda)^{-1}f\|.
\end{split}
\ee
Since $\sigma\in[-\frac{1}{2},\frac{1}{2}]$ and $a(\eps,\sigma),b(\eps,\sigma)\to0$ as $\eps\to0$, using the standard resolvent estimates for the operator $B_0(\sigma)$ (cf. \cite[Lemma 4.3]{EV87}, \cite[Theorem VIII.17]{RS80}),  we conclude that the real part of the spectrum of $B(\eps, \omega, s,\sigma)$ has negative upper bound for $\sigma \in [-\frac{1}{2},\frac{1}{2})\setminus \Gamma $, where $\Gamma= \{ \sigma | -\eta <\sigma < \eta \}$, $\eta$ is sufficiently small.

 From now on, we consider the following ``dangerous set''
 \be
 \Gamma= \{ \sigma | -\eta <\sigma < \eta \}
 \ee
 for some sufficiently small $\eta >0$.
\bigskip

\subsection{Stability of the bifurcating periodic solutions: coperiodic case $\sigma=0$}

We now consider the eigenvalue problem of $B(\eps, \omega, s,0)$:
\be
0=\Big[ B(\eps, \omega, s,0)-\l I \Big] W.
\ee
In order to use the Lyapunov-Schmidt reduction, we decompose $W = \b_1 U_{1}+\b_{2} U_{2} + \mathcal V$ and we first solve
\be\label{lambda}
\begin{split}
	0
	& =(I-Q)\Big[ B(\eps, \omega, s,0)-\l I \Big] (\b_1 U_{1}+\b_{2} U_{2}+\b_{3} U_{3} + \mathcal V),
\end{split}
\ee

where
\be
(I-Q)\Big[ B(\eps, \omega, s,0)-\l I \Big] (\b_1 U_{1}+\b_{2} U_{2}+\b_{3} U_{3} + \mathcal V):\mathbb{R}\times\mathbb{C}\times\mathbb{R}^3\times\ran(I-P)\to\ran(I-Q).
\ee

By the Implicit Function Theorem, there exists  an open neighborhood $U\subset\mathbb{R}\times\mathbb{C}\times\mathbb{R}^3$ of $(0,0,0,0,0)$ and a unique function $\mathcal V: U\to(I-P)H^6_{per}([0,2\pi], \RR)$ that solves  \eqref{lambda} for $(\eps,\lambda,\b_1,\b_2,\b_3)\in U$.

Next, it is clear that the relation between $\b$ and $\mathcal V$ is linear. Then, let $\mathcal V(\eps, \omega, s,\lambda,\b)=\mathcal V_1(\eps, \omega, s,\lambda)\b_1+\mathcal V_2(\eps, \omega, s,\lambda)\b_2+\mathcal V_3(\eps, \omega, s,\lambda)\b_3$. Now let us find asymptotic expansions of $\mathcal V_1$, $\mathcal V_2$ and $\mathcal V_3$ with respect to parameter $\eps$.\\
1.) First, we compute $\mathcal V_i(0,\omega,\lambda)=\partial_{\b_i}\mathcal V|_{\eps=0}$. We differentiate \eqref{lambda} with respect $\b_i$ and plug in $0$ for $\eps$.

\be
\begin{split}
	0
	& =(I-Q)\Big[ B(0,\omega,0)-\l I \Big] (U_{i} + \partial_{\b_i}\mathcal V|_{\eps=0}).
\end{split}
\ee
Notice that $ B(0,\omega,0)  U_{i}=L_{per}U_{i}=0$ and $(I-Q)U_i=0$. Since $(I-Q)\Big[ B(0,\omega,0)-\l I \Big](I-P)$ is invertible for small values of $\lambda$, we conclude that
\be
\mathcal V_i(0,\omega,\lambda)=\partial_{\b_i}\mathcal V|_{\eps=0}=0.
\ee
2.) Now, we differentiate the second equation of \eqref{lambda} with respect to $\b_1$ and $\eps$ and, then, plug in $0$ for $\eps$

\be
\begin{split}
	0
	& =(I-Q)\partial_{\eps}B(0,\omega,0)U_{1} + (I-Q)(B(0,\omega,0)-\lambda)\partial_{\eps}\partial_{\b_1}\mathcal V|_{\eps=0},
\end{split}
\ee
or
\be
\begin{split}
	(I-Q)(B(0,\omega,0)-\lambda)\partial_{\eps}\partial_{\b_1}\mathcal V|_{\eps=0}
	& =-(I-Q)\partial_{\eps}B(0,\omega,0)U_{1}.
\end{split}
\ee
Taking into account formulas \eqref{linearization about tilde u} and \eqref{Bloch}, we arrive at
\be\label{ebetta}
\begin{split}
	&-(I-Q)\partial_{\eps}B(0,\omega,0)U_{1}=-(I-Q)(-\partial_\xi^2\big[ -(1+\partial_\xi^2)^2+\{-2\omega(1+\partial_\xi^2)^2-4\omega\partial_\xi^2(1+\partial_\xi^2)\\
	&-12s\sqrt{\frac{1-4\omega^2}{27-2s^2}}\cos\xi\}\big])U_1=-12s\sqrt{\frac{1-4\omega^2}{27-2s^2}}(I-Q) \partial^2_\xi (\cos^2\xi)=24s\sqrt{\frac{1-4\omega^2}{27-2s^2}}\cos 2\xi.
\end{split}
\ee
Since $span\{\cos 2\xi\}$ is an invariant subspace for the invertible operator $(I-Q)(B(0,\omega,0)-\lambda)(I-P)$, $\partial_{\eps}\partial_{\b_1}\mathcal V|_{\eps=0}$ is of the form $a \cos 2\xi$.\\
Next, note that
\be
\begin{split}\label{h1}
&(I-Q)(B(0,\omega,0)-\lambda)(a\cos 2\xi)=a(I-Q)\big(\partial_\xi^2(1+\partial_\xi^2)^2-\lambda\big) (\cos 2\xi)\\
&=a(-36-\lambda)\cos 2\xi
\end{split}
\ee

Using \eqref{ebetta}-\eqref{h1}, we derive that
$
a(-36-\lambda)=24s\sqrt{\frac{1-4\omega^2}{27-2s^2}}.
$
Hence,
\be
\partial_{\eps}\partial_{\b_1}\mathcal V|_{\eps=0}=\frac{24s}{-36-\lambda}\sqrt{\frac{1-4\omega^2}{27-2s^2}}  \cos 2\xi.
\ee
Similarly,
$
\partial_{\eps}\partial_{\b_2}\mathcal V|_{\eps=0}=\frac{24s}{-36-\lambda}\sqrt{\frac{1-4\omega^2}{27-2s^2}}  \sin 2\xi$, and it is also clear that $\partial_{\eps}\partial_{\b_3}\mathcal V|_{\eps=0}=0$.

So far, we have shown that 
$$
\mathcal V= (\frac{24s}{-36-\lambda}\sqrt{\frac{1-4\omega^2}{27-2s^2}}  \cos 2\xi\eps+\mathcal O(\eps^2))\b_1+ (\frac{24s}{-36-\lambda}\sqrt{\frac{1-4\omega^2}{27-2s^2}}  \sin 2\xi  \eps+\mathcal O(\eps^2))\b_2+\mathcal O(\eps^2)\b_3.
$$
3.) Now, we would like to compute $\partial^2_{\eps}\partial_{\b_i}\mathcal V|_{\eps=0}$. Differentiating the second equation of \eqref{lambda} with respect to $\b_1$ and $\eps$ twice and, then, plugging in $0$ for $\eps$, we obtain
$$
	0
	 =(I-Q)\partial^2_{\eps}B(0,\omega,0)U_{1} + 2(I-Q)\partial_{\eps}B(0,\omega,0)\partial_{\eps}\partial_{\b_1}\mathcal V|_{\eps=0}+(I-Q)(B(0,\omega,0)-\lambda)\partial^2_{\eps}\partial_{\b_1}\mathcal V|_{\eps=0},
$$
or
$$
	(I-Q)(B(0,\omega,0)-\lambda)\partial^2_{\eps}\partial_{\b_1}\mathcal V|_{\eps=0}
	 =-(I-Q)\partial^2_{\eps}B(0,\omega,0)U_{1} - 2(I-Q)\partial_{\eps}B(0,\omega,0)\partial_{\eps}\partial_{\b_1}\mathcal V|_{\eps=0}.
$$
Therefore,
\be
\begin{split}
	&(I-Q)(B(0,\omega,0)-\lambda)\partial^2_{\eps}\partial_{\b_1}\mathcal V|_{\eps=0}
	 =-(I-Q)\partial^2_{\eps}B(0,\omega,0)U_{1} - 2(I-Q)\partial_{\eps}B(0,\omega,0)\partial_{\eps}\partial_{\b_1}\mathcal V|_{\eps=0}\\
	&=-2(I-Q)\big\{-\partial_\xi^2\big[-8\omega^2\partial_\xi^2(1+\partial_\xi^2)-4\omega^2\partial_\xi^4+1-\frac{54(1-4\omega^2)}{27-2s^2}+8\omega s(14s^2-81)\sqrt{\frac{1-4\omega^2}{(27-2s^2)^3}}\cos\xi\\
	&-2(1-4\omega^2)\cos 2\xi\big]U_1-\partial_\xi^2\big[ -(1+\partial_\xi^2)^2+\{-\omega(1+\partial_\xi^2)^2-2\omega\partial_\xi^2(1+\partial_\xi^2)\\
	&-12s\sqrt{\frac{1-4\omega^2}{27-2s^2}}\cos\xi\}\big]\partial_{\eps}\partial_{\b_1}\mathcal V|_{\eps=0}\big\}=(*)\cos2\xi+(*)\cos3\xi.
\end{split}
\ee		
Hence, $\partial^2_{\eps}\partial_{\b_1}\mathcal V|_{\eps=0}$ is of the form $(*)\cos2\xi+(*)\cos3\xi$. Similarly, $\partial^2_{\eps}\partial_{\b_2}\mathcal V|_{\eps=0}$ is of the form $(*)\sin2\xi+(*)\sin3\xi$, and $\partial^2_{\eps}\partial_{\b_3}\mathcal V|_{\eps=0}$  is of the form $(*)\cos2\xi$. \\
 Therefore,

	\be\label{Ve2}
	\begin{split}
\mathcal V&= (\frac{24s}{-36-\lambda}\sqrt{\frac{1-4\omega^2}{27-2s^2}}  \cos 2\xi\eps+[(*)\cos2\xi+(*)\cos3\xi]\eps^2+\mathcal O(\eps^3))\b_1\\
&+ (\frac{24s}{-36-\lambda}\sqrt{\frac{1-4\omega^2}{27-2s^2}}  \sin 2\xi\eps+  [(*)\cos2\xi+(*)\cos3\xi]\eps^2+\mathcal O(\eps^3))\b_2+((*)\cos 2\xi\eps^2+\mathcal O(\eps^3))\b_3.
\end{split}
\ee

In order to obtain the reduced system for the spectral problem, we substitute $W = \b_1 U_{1}+\b_{2} U_{2} +\b_3 U_{3}+ \mathcal V$, where $\mathcal V$ is given by \eqref{Ve2} into the equation 

\be\label{sred}
\begin{split}
	0
	& =\tilde Q\Big[ B(\eps, \omega, s,0)-\l I \Big] W,
\end{split}
\ee
that is,
\be
\begin{split}
	0
	& =\tilde Q\Big[ B(\eps, \omega, s,0)-\l I \Big](\b_1 U_{1}+\b_{2} U_{2}+\b_3 U_{3})+\tilde Q\Big[ B(\eps, \omega, s,0)-\l I \Big]\mathcal V\\
	&=\big[\bp M_{11} & M_{12}&M_{13}\\ M_{21} & M_{22}& M_{23}\\
	M_{31}&M_{32} &M_{33}\ep+\bp \tilde M_{11} & \tilde M_{12}&\tilde M_{13}\\\tilde M_{21} &\tilde M_{22}&\tilde M_{23}\\
	\tilde M_{31}&\tilde M_{32} &\tilde M_{33}\ep+
	remainder\big]\bp \b_1 \\ \b_2\\ \b_3\ep,
\end{split}
\ee
where 
\be
\begin{split}
M_{11}=&\big[(1-4\omega^2)-\frac{54(1-4\omega^2)}{27-2s^2}-(1-4\omega^2)+\frac{4s^2(1-4\omega^2)}{27-2s^2}\big]\eps^2-\lambda,\,\,M_{12}=0,\\
M_{13}=&-12s\sqrt{\frac{1-4\omega^2}{27-2s^2}}\eps+8\omega s(14s^2-81)\sqrt{\frac{1-4\omega^2}{(27-2s^2)^3}}\eps^2,\,\,M_{21}=0, \\
 M_{22}=&\big[(1-4\omega^2)-\frac{54(1-4\omega^2)}{27-2s^2}+(1-4\omega^2)+\frac{4s^2(1-4\omega^2)}{27-2s^2}\big]\eps^2-\lambda,\,\,M_{23}=0,\\
 M_{31}=&M_{32}=0,\,\,M_{32}=-\lambda,\\
 \tilde M_{11}=&\frac{-144s^2(1-4\omega^2)}{(-36-\lambda)(27-2s^2)}\eps^2, \,\,\tilde M_{12}=0,\,\,\tilde M_{13}=0,\\
 \tilde M_{21}=& 0,\,\,\tilde M_{22}=\frac{-144s^2(1-4\omega^2)}{(-36-\lambda)(27-2s^2)}\eps^2, \,\,\tilde M_{23}=0,\,\,M_{3i}=0,\\
 remainder=&\mathcal O(\eps^3).
\end{split}
\ee
Or,
\be\label{spmatrix}
\begin{split}
	0
	&=\bp -2(1-4\omega^2)\eps^2-\lambda & 0 &-12s\sqrt{\frac{1-4\omega^2}{27-2s^2}}\eps+8\omega s(14s^2-81)\sqrt{\frac{1-4\omega^2}{(27-2s^2)^3}}\eps^2\\ 0 & -\lambda&0\\
	0&0&-\lambda \ep\bp \b_1 \\ \b_2 \\\b_3\ep\\
	&+
	\bp \mathcal O(\eps^2(\eps+|\lambda|)) & \mathcal O(\eps^3) & \mathcal O(\eps^3)\\ \mathcal O(\eps^3) & \mathcal O(\eps^2(\eps+|\lambda|))& \mathcal O(\eps^3)\\ \mathcal O(\eps^3) & \mathcal O(\eps^3)& \mathcal O(\eps^3)\ep\bp \b_1 \\ \b_2\\\b_3  \ep.
\end{split}
\ee
Now, we will establish the following refined remainder estimate.
\begin{lemma}\label{remainder}
The remainder in \eqref{spmatrix} has the form
\be
\begin{split}
	\bp \mathcal O(\eps^2(\eps+|\lambda|)) & \mathcal O(\eps^3) & \mathcal O(\eps^3)\\ \mathcal O(\eps^3) & \mathcal O(\eps^2(\eps+|\lambda|))& \mathcal O(\eps^3)\\ \mathcal O(\eps^3) & \mathcal O(\eps^3)& \mathcal O(\eps^3)\ep=\bp \mathcal O(\eps^2(\eps+|\lambda|)) & \mathcal O(\eps^3|\lambda|) & \mathcal O(\eps^3(1+|\lambda|))\\  \mathcal O(\eps^3|\lambda|) & \mathcal O(\eps^2|\lambda|)&  \mathcal O(\eps^3|\lambda|)\\  \mathcal O(\eps^3|\lambda|) &  \mathcal O(\eps^3|\lambda|)&  \mathcal O(\eps^3|\lambda|)\ep
\end{split}
\ee
\end{lemma}
\begin{proof}
	All we need to show is that if $\lambda=0$, then the reduced spectral system $\tilde Q B(\eps, \omega, s,0) W$ is of the form
	\be\label{m(eps)}
	\begin{split}
		\bp -2(1-4\omega^2)\eps^2+\mathcal O(\eps^3) & 0 &-12s\sqrt{\frac{1-4\omega^2}{27-2s^2}}\eps+8\omega s(14s^2-81)\sqrt{\frac{1-4\omega^2}{(27-2s^2)^3}}\eps^2+\mathcal O(\eps^3)\\ 0 & 0&0\\
		0&0&0 \ep\bp \b_1 \\ \b_2 \\\b_3\ep\\.
	\end{split}
	\ee
	Now, we plug $0$ for $\lambda$ in \eqref{lambda} and then differentiate it with respect to $\b_1$.
	
	\be
	\begin{split}
		0
		& =(I-Q) B(\eps, \omega, s,0)( U_{1} + \partial_{\b_1}\mathcal V).
	\end{split}
	\ee
	Let us also  differentiate the second equation of \eqref{projection equations} with respect to $\alpha_1$ and then plug in $0$ for $\a_2$ and $\a_3$.
	\be
	\begin{split}
		& (I-Q)[-k^2\partial_\xi^2\big([ -(1+k^2\partial_\xi^2)^2+\epsilon^2] (U_1+\partial_{\alpha_1}V)\\
		&-(2s u+ 3u^2)(U_1+\partial_{\alpha_1}V)  \big)]=0,
	\end{split}
	\ee
	where $u=\a_1U_1+V$.
	Due to the uniqueness part in the Implicit Function Theorem, we conclude that 
	 \be
	 	\begin{split}
	 		\partial_{\b_1}\mathcal V(\eps, \omega, s,0,\b)=\partial_{\alpha_1}V|_{\alpha_2=\alpha_3=0}=\big(\frac{s}{ -(1-4k^2)^2+\epsilon^2}\cos 2\xi\big)\alpha_1+\mathcal{O}(|\a_1|^2).
	 	\end{split}
	 	\ee
	 	Similarly, we conclude that 
	 	\be
	 		 	\begin{split}
	 		 		\partial_{\b_2}\mathcal V(\eps, \omega, s,0,\b)=\partial_{\alpha_2}V|_{\alpha_2=\alpha_3=0}=\big(\frac{s}{ -(1-4k^2)^2+\epsilon^2}\sin 2\xi\big)\alpha_1+\mathcal{O}(|\a_1|^2).
	 		 	\end{split}
	 		 	\ee
	 		 	Therefore, in order to find the entries $(1,1)$, $(1,2)$ $(2,1)$ $(2,2)$ of the spectral matrix from the reduced system \eqref{sred} we differentiate the first and second equations of \eqref{reduced} with respect to $\a_1$ and $\a_2$ respectively, and then plug $0$ for $\a_2$ and $\a_3$. The trivial part of reduced system \eqref{reduced} implies that entries $(3,i)$ are equal to $0$. Also, one can compute entries $(1,3)$ and $(2,3)$ by differentiating the first and second equations of \eqref{reduced} with respect to $\a_3$ and then pluging $0$ for $\a_2$ and $\a_3$.
	 		 	  Hence,
	 		\be
	 		\begin{split}
	 		\tilde Q B(\eps, \omega, s,0) W=m\bp \b_1 \\ \b_2 \\\b_3\ep,
	 		\end{split}
	 		\ee 	  
	 		where 
	 			\be
	 			\begin{split}
	 			m_{11}&= [(-(1-k^2)^2+\eps^2] k^2  - 
	 				3(\frac{3}{4}+sf)k^2\a_1^2  +\mathcal{O}(|\a_1|^4),\, m_{12}=0,\\
	 				m_{13}&=-2s k^2\a_1  +\mathcal{O}(|\a_1|^3), \,\,m_{21}=0,\\
	 				m_{22}&= [(-(1-k^2)^2+\eps^2] k^2  - 
	 				(\frac{3}{4}+sf)k^2\a_1^2  +\mathcal{O}(|\a_1|^4),\, m_{23}=0,\,\,m_{3i}=0.\\
	 			\end{split}
	 			\ee

	 		Note that $\frac{1}{\frac{3}{4}+sf}m_{22}=\mathcal{A}-\a^2  +\mathcal{O}(|\a|^4)$ is exactly the left-hand side of \eqref{short}. Using formulas \eqref{Acal} and \eqref{alpha}, we arrive at formula \eqref{m(eps)} for the reduced spectral system $\tilde Q B(\eps, \omega, s,0) W$.
	 	
\end{proof}

Using the refined remainder estimate, we obtain the following characterization of co-periodic stability.

\begin{proposition}[Co-periodic stability]\label{coperprop}
	Let $u_{\eps, \omega, s}$ be the solution from Theorem \ref{existence}.  There exist
	 $\tilde\eps_0\in(0,\eps_0]$, where $\eps_0$ is taken from Theorem \ref{existence}, and $\delta>0$ such that for all $\eps\in[0,\tilde\eps_0)$ and all $\omega\in[-\frac{1}{2},\frac{1}{2}]$ the spectrum of $B(\eps, \omega, s,0)$ has the decomposition:
	\be
	\begin{split}
		\Sp(B(\eps, \omega, s,0))=S\cup\{\lambda_1,\lambda_2\},
	\end{split}
	\ee
	where
	\be
	\begin{split}
		\lambda_1(\eps, \omega, s)&=-2(1-4\omega^2)\eps^2+\mathcal O(\eps^3),\\
		\lambda_2(\eps, \omega, s)&=0,\\
		\lambda_3(\eps, \omega, s)&=0.
	\end{split}
	\ee
Moreover, if $\lambda\in S$, then $\Re\lambda<-\delta$.
\end{proposition}

\begin{proof}
	Setting the determinant of the matrix from \eqref{spmatrix} equal to $0$ and using Lemma \ref{remainder}, we obtain
	\be
	 	\begin{split}
	 		 &\big(c(\eps, \omega, s)-\lambda+ \mathcal O(\eps^2|\lambda|)\big)[(-\lambda+\mathcal O(\eps^2|\lambda|))(-\lambda+\mathcal O(\eps^2|\lambda|))-\mathcal O(\eps^6|\lambda|^2)]-\mathcal O(\eps^3|\lambda|)\\
	 		 &\times[\mathcal O(\eps^3|\lambda|)(-\lambda+\mathcal O(\eps^3|\lambda|))-\mathcal O(\eps^6|\lambda|^2)]+\mathcal O(\eps+\eps^3|\lambda|)[\mathcal O(\eps^6|\lambda|^2)-\mathcal O(\eps^3|\lambda|)(-\lambda+\mathcal O(\eps^2|\lambda|))]=0,
	 	\end{split}
	\ee
	where $c(\eps)=-2(1-4\omega^2)\eps^2+\mathcal O(\eps^3)$. Or,
	\be\label{pol}
	\begin{split}
	\lambda^3-\lambda^2\big(c(\eps)+\mathcal O(\eps^2|\lambda|)\big)+\lambda\mathcal O(\eps^4|\lambda|)+\mathcal O(\eps^6|\lambda|^2)=0.
	\end{split}
	\ee
	After factorization we arrive at 
	\be
	\begin{split}
		\lambda^2(\lambda-\tilde c(\eps)+\mathcal O(\eps^2|\lambda|))=0.
	\end{split}
	\ee
	where $\tilde c(\eps)=-2(1-4\omega^2)\eps^2+\mathcal O(\eps^3)$. 
Therefore,
		\be
		\begin{split}
			\lambda_{1}&=\tilde c(\eps)+\mathcal O(\eps^2|\lambda|),\\
			\lambda_{2}&=0,\\
			\lambda_{3}&=0.
		\end{split}
		\ee
\end{proof}

\subsection{Stability of the bifurcating periodic solutions: general case}

We now consider the eigenvalue problem of $B(\eps, \omega, s,\sigma)$:
\be\label{evp}
0=\Big[ B(\eps, \omega, s,\sigma)-\l I \Big] W.
\ee
In order to use the Lyapunov-Schmidt reduction, we decompose $W = \b_1 U_{1}+\b_{2} U_{2}+\b_{3} U_{3} + \mathcal V$ and we first solve
\be\label{slambda}
\begin{split}
	0
	& =(I-Q)\Big[ B(\eps, \omega, s,\sigma)-\l I \Big] (\b_1 U_{1}+\b_{2} U_{2}+\b_{3} U_{3} + \mathcal V).
\end{split}
\ee

Next, we go through steps described in the previous section. 
Next, it is clear that the relation between $\b$ and $\mathcal V$ is linear. Then, let $\mathcal V(\eps, \omega, s,\sigma,\lambda,\b)=\mathcal V_1(\eps, \omega, s,\sigma,\lambda)\b_1+\mathcal V_2(\eps, \omega, s,\sigma,\lambda)\b_2+\mathcal V_3(\eps, \omega, s,\sigma,\lambda)\b_3$. Now let us find asymptotic expansions of $\mathcal V_1$, $\mathcal V_2$ and $\mathcal V_3$ with respect to parameter $\eps$.\\
1.) First, we compute $\mathcal V_i(0,\omega,\sigma,\lambda)=\partial_{\b_i}\mathcal V|_{\eps=0}$. \\
a.) Differentiating \eqref{slambda} with respect $\b_i$ and plugging in $0$ for $\eps$, we obtain
$$
	0  =(I-Q)\Big[ B(0,\omega,\sigma)-\l I \Big] (U_{i} + \partial_{\b_i}\mathcal V|_{\eps=0}).
$$
Therefore,
\be\label{eqV}
\begin{split}
	&(I-Q)\Big[ B(0,\omega,\sigma)-\l I \Big]\partial_{\b_i}\mathcal V|_{\eps=0}
	=-(I-Q)\Big[ B(0,\omega,\sigma)-\l I \Big] U_{i}\\
	&=-(I-Q)\big[(\partial_\xi+i\sigma)^2(1+(\partial_\xi+i\sigma)^2)^2-\lambda\big]U_i=0.
\end{split}
\ee
We conclude that $\mathcal V_1(0,\omega,\sigma,\lambda)=0$. 

2.) Next, we would like to compute $\partial_{\eps}\partial_{\b_i}\mathcal V|_{\eps=0}$.\\
a.) We start with $\partial_{\eps}\partial_{\b_1}\mathcal V|_{\eps=0}$.
Differentiating \eqref{slambda} with respect $\b_1$ and $\eps$, and plugging in $0$ for $\eps$, we obtain
$
	0
	 =(I-Q)\big\{\partial_{\eps}B(0,\omega,\sigma) (U_{1} + \partial_{\b_1}\mathcal V|_{\eps=0})+[ B(0,\omega,\sigma)-\l I]\partial_{\eps}\partial_{\b_1}\mathcal V|_{\eps=0}\big\}.
$

Therefore,
$
	(I-Q)[ B(0,\omega,\sigma)-\l I]\partial_{\eps}\partial_{\b_1}\mathcal V|_{\eps=0}
	 =-(I-Q)\partial_{\eps}B(0,\omega,\sigma) (U_{1} + \partial_{\b_1}\mathcal V|_{\eps=0}),
$
or, using \eqref{linearization about tilde u},
\be
\begin{split}\label{Vsel}
	&(I-Q)[ B(0,\omega,\sigma)-\l I]\partial_{\eps}\partial_{\b_1}\mathcal V|_{\eps=0}
	=(I-Q)(\partial_\xi+i\sigma)^2\big[ -2\omega(1+(\partial_\xi+i\sigma)^2)^2\\
	&-4\omega(\partial_\xi+i\sigma)^2(1+(\partial_\xi+i\sigma)^2)-12s\sqrt{\frac{1-4\omega^2}{27-2s^2}}\cos\xi] U_{1}=-12s\sqrt{\frac{1-4\omega^2}{27-2s^2}}(I-Q)(\partial_\xi+i\sigma)^2\cos^2\xi\\&=-12s\sqrt{\frac{1-4\omega^2}{27-2s^2}}[(-2-\frac{\sigma^2}{2})\cos2\xi-2i\sigma\sin2\xi].
\end{split}
\ee
Hence, $\partial_{\eps}\partial_{\b_1}\mathcal V|_{\eps=0}$ is of the form
\be
 \begin{split}
&\partial_{\eps}\partial_{\b_1}\mathcal V|_{\eps=0}= (*)\sin 2\xi+(*)\cos 2\xi.
\end{split}
 \ee

b.) In a similar fashion, one can show that $\partial_{\eps}\partial_{\b_2}\mathcal V|_{\eps=0}=(*)\sin 2\xi+(*)\cos 2\xi$ and $\partial_{\eps}\partial_{\b_3}\mathcal V|_{\eps=0}=0$.
Overall, we have
\be
 \begin{split}\label{Vse}
\mathcal V(\eps, \omega, s,\sigma,\lambda,\b)&= \big(((*)\sin 2\xi+(*)\cos 2\xi)\eps+\mathcal O(\eps^2)\big)\b_1+\big(((*)\sin 2\xi+(*)\cos 2\xi)\eps+\mathcal O(\eps^2)\big)\b_2\\
&+\mathcal O(\eps^2)\b_3.
\end{split}
 \ee

In order to obtain the reduced system for the spectral problem, we plug $W = \b_1 U_{1}+\b_{2} U_{2} + \mathcal V$, where $\mathcal V$ is given by \eqref{Vse},
 into the equation 
\be\label{redeq}
\begin{split}
	0
	& =\tilde Q\Big[ B(\eps, \omega, s,\sigma)-\l I \Big] W.
\end{split}
\ee

Using \eqref{linearization about tilde u} and \eqref{Vse}, we arrive at
\be\label{spmatrixs}
\begin{split}
	0
	& =m\bp \b_1 \\ \b_2\\\b_3 \ep:=\tilde Q\Big[ B(\eps, \omega, s,\sigma)-\l I \Big](\b_1 U_{1}+\b_{2} U_{2}+\b_{3} U_{3})+\tilde Q\Big[ B(\eps, \omega, s,\sigma)-\l I \Big]\mathcal V\\
	&=\bp -4\sigma^2+c(\eps)-\lambda & 8i\omega\sigma\eps &-12s\sqrt{\frac{1-4\omega^2}{27-2s^2}}\eps\\ -8i\omega\sigma\eps & -4\sigma^2-\lambda&0\\
	-6s\sqrt{\frac{1-4\omega^2}{27-2s^2}}\sigma^2\eps&0&-\sigma^2-\lambda \ep\bp \b_1 \\ \b_2\\\b_3 \ep+O\bp \b_1 \\ \b_2\\\b_3 \ep,
	\end{split}
	\ee
	where 
	\be
	\begin{split}
	c(\eps)&=-2(1-4\omega^2)\eps^2+\mathcal O(\eps^3),\\
	O_{11}&=\mathcal O(\sigma^4+\sigma^2\eps+(|\lambda|(1+\sigma)+\sigma)\eps^2),\,\,O_{12}=\mathcal O(|\sigma|^3+|\sigma|^3\eps+|\sigma|\eps^2+|\lambda|\eps^3),\\
	O_{13}&=\mathcal O(\eps^2(1+|\sigma|)+|\lambda|\eps^3),\,\,O_{21}=\mathcal O(|\sigma|^3+|\sigma|^3\eps+|\sigma|\eps^2+|\lambda|\eps^3),\\
	O_{22}&=\mathcal O(\sigma^4+\sigma^2\eps+(|\lambda|(1+|\sigma|)+\sigma)\eps^2),\,\,O_{23}=\mathcal O(\eps^2|\sigma|+|\lambda|\eps^3),\\
	O_{31}&=\mathcal O((1+|\lambda|)|\sigma|\epsilon^2+|\lambda|\eps^3),\,\,O_{32}=\mathcal O((1+|\lambda|)|\sigma|\epsilon^2+|\lambda|\eps^3),\\
	O_{33}&=\mathcal O(\sigma^4+\sigma^2\eps+|\sigma|\eps^2+|\lambda|\eps^3).
		\end{split}
		\ee

One can improve the error estimates in \eqref{spmatrixs} using symmetric properties of the eigenvalue problem \eqref{evp}. In particular, we gain additional information on elements of matrix $m$.
\begin{lemma}
The diagonal elements of matrix $m$ and elements $m_{13}$, $m_{31}$  in \eqref{spmatrixs} are even in $\sigma$ and all other elements are odd in $\sigma$. Moreover, if $\lambda$ is real then $m_{ii}$, $m_{13}$ and $m_{31}$ are real-valued while all other elements are purely imaginary. 
\end{lemma}
\begin{proof}
First, we note that \eqref{evp} possesses two symmetries \cite{M2}
\be\label{symms}
\begin{split}
&[ B(\eps, \omega, s,\sigma)-\l I \Big]R_1=R_1[ B(\eps, \omega, s,-\sigma)-\l I \Big],\,\lambda\in\mathbb{C}\\
&[ B(\eps, \omega, s,\sigma)-\l I \Big]R_2=R_2[ B(\eps, \omega, s,-\sigma)-\l I \Big],\,\lambda\in\mathbb{R}.
\end{split}
\ee
where $R_1W(\xi)=W(-\xi)$, and $R_2W(\xi)=\overline W(\xi)$. \\
Following the steps of proof of Proposition 3.3 \cite[Chapter VII]{GS}, one can show that the reduced system \eqref{redeq} commutes with symmetries defined in \eqref{symms}, i.e.
\begin{equation}
m(\eps, \omega, s,\sigma,\lambda)\bp \b_1 \\ -\b_2\\\b_3 \ep=\bp m_{11}(\eps, \omega, s,-\sigma,\lambda) &m_{12}(\eps, \omega, s,-\sigma,\lambda)&m_{13}(\eps, \omega, s,-\sigma,\lambda)\\ -m_{21}(\eps, \omega, s,-\sigma,\lambda)&-m_{22}(\eps, \omega, s,-\sigma,\lambda)&-m_{23}(\eps, \omega, s,-\sigma,\lambda)\\  m_{31}(\eps, \omega, s,-\sigma,\lambda) &m_{32}(\eps, \omega, s,-\sigma,\lambda)&m_{33}(\eps, \omega, s,-\sigma,\lambda) \ep\bp \b_1 \\ \b_2 \\\b_3\ep,
\end{equation}
where $\lambda\in\mathbb{C}$.
And
\begin{equation}
m(\eps, \omega, s,\sigma,\lambda)\bp \b_1 \\ \b_2\\\b_3 \ep=\overline m(\eps, \omega, s,-\sigma,\lambda)\bp \b_1 \\ \b_2\\\b_3 \ep,\,\lambda\in\mathbb{R}.
\end{equation}
\end{proof}
\begin{corollary}\label{errcorr}
The error matrix in \eqref{spmatrixs} has the form
\begin{equation*}
\bp \mathcal O(\sigma^4+\sigma^2\eps+|\lambda|(1+\sigma^2)\eps^2) & \mathcal O(\sigma^3+|\sigma|^3\eps+|\sigma|\eps^2+|\lambda\sigma|\eps^3)& \mathcal O(\eps^2(1+\sigma^2)+|\lambda|\eps^3)\\  \mathcal O(|\sigma|^3+|\sigma|^3\eps+|\sigma|\eps^2+|\lambda\sigma|\eps^3) & \mathcal O(\sigma^4+\sigma^2\eps+|\lambda|(1+\sigma^2)\eps^2)&\mathcal O(\eps^2|\sigma|+|\lambda\sigma|\eps^3)\\ \mathcal O((1+|\lambda|)\sigma^2\epsilon^2+|\lambda|\eps^3)&\mathcal O((1+|\lambda|)|\sigma|\epsilon^2+|\lambda\sigma|\eps^3)&\mathcal O(\sigma^4+\sigma^2\eps+|\lambda|\eps^3) \ep.
\end{equation*}
\end{corollary}

\begin{proof}[Proof of Theorem \ref{stabthm}]
	Let us take the determinant of $m$ from formula \eqref{spmatrixs}. 
	\be
	\begin{split}
	\det	m(\eps, \omega, s,\sigma,\lambda)&:=-\lambda^3+\lambda^2\big(c(\eps)-9\sigma^2+\mathcal O(\sigma^4+\sigma^2\eps+|\lambda|(1+\sigma^2)\eps^2)\big)\\
		&-\lambda\big(24\sigma^4-5c(\eps)\sigma^2-72s^2A\sigma^2\eps^2-64\omega^2\sigma^2\eps^2\\
		&+\mathcal O(\sigma^6+\sigma^4\eps+(\sigma^2+|\lambda|)\sigma^2\eps^2+(1+|\lambda|)\sigma^2\eps^3+|\lambda|\eps^4)\big)\\
		&-16\sigma^6+4c(\eps)\sigma^4+64\omega^2\sigma^4\eps^2+288s^2A\sigma^4\eps^2\\
		&+\mathcal O(\sigma^8+\sigma^6\eps+(\sigma^2+|\lambda|)\sigma^4\eps^2+(1+|\lambda|)\sigma^4\eps^3+|\lambda|^2\eps^6),
	\end{split}
	\ee
	where $A=\frac{1-4\omega^2}{27-2s^2}$ and $c(\eps)=-2(1-4\omega^2)\eps^2+\mathcal O(\eps^3)$. Notice that the error terms are real.
	  According to the Weierstrass Preparation Theorem, there exists an analytic function $q(\eps,\sigma,\lambda)$ in a neighborhood of $(0,0,0)$ such that $q(0,0,0)=-1$ and
	 	\be\label{wpt}
	 	\begin{split}
	 		q(\eps,\sigma,\lambda)\det m(\eps, \omega, s,\sigma,\lambda)=\lambda^3+a_2\lambda^2+a_1\lambda+a_0.
	 	\end{split}
	 	\ee
	 	Notice that
	 	\be
	 	\begin{split}
	 		a_0(\eps,\sigma)&=q(\eps,\sigma,0)\det m(\eps, \omega, s,\sigma,0)=-\det m(\eps, \omega, s,\sigma,0)+\mathcal O\big(\sigma^4(\eps+|\sigma|)^3\big)\\
	 		&=16\sigma^6-4c(\eps)\sigma^4-64\omega^2\sigma^4\eps^2-288s^2A\sigma^4\eps^2+\mathcal O\big(\sigma^4(\eps+|\sigma|)^3\big),\\
	 		a_1(\eps,\sigma)&=q'_{\lambda}(\eps,\sigma,0)\det m(\eps, \omega, s,\sigma,0)
	 		+q(\eps,\sigma,0)(\det m)'_{\lambda}(\eps, \omega, s,\sigma,0)\\
	 		&=-(\det m)'_{\lambda}(\eps, \omega, s,\sigma,0)+\mathcal O(\sigma^2(\eps+|\sigma|)^3)=24\sigma^4-5c(\eps)\sigma^2-72s^2A\sigma^2\eps^2-64\omega^2\sigma^2\eps^2\\
	 		&+\mathcal O(\sigma^2(\eps+|\sigma|)^3),\\
	 		a_2(\eps,\sigma)&=-\frac{1}{2}(\det m)''_{\lambda}(\eps, \omega, s,\sigma,0)+\mathcal O((\eps+\sigma)^3)=-c(\eps)+9\sigma^2+\mathcal O((\eps+|\sigma|)^3).
	 	\end{split}
	 	\ee
	 
	 	Therefore, the eigenvalue problems boils down to the third order polynomial
	 	\be\label{cubic}
	 	\begin{split}
	 		&\lambda^3+\big(-c(\eps)+9\sigma^2+\mathcal O((\eps+|\sigma|)^3)\big)\lambda^2-\sigma^2\big(-24\sigma^2+5c(\eps)+(72s^2A+64\omega^2)\eps^2+\mathcal O((\eps+|\sigma|)^3)\big)\lambda\\
	 		&-\sigma^4(-16\sigma^2+4c(\eps)+(288s^2A+64\omega^2)\eps^2+\mathcal O\big((\eps+|\sigma|)^3\big))=0.
	 	\end{split}
	 	\ee
	 	Since the error terms in the $\det m$ are real, they are in \eqref{cubic} as well. In particular, if $\omega=\pm\frac{1}{2}$ the product of the eigenvalues $\sigma^4(-16\sigma^2+4c(\eps)+(288s^2A+64\omega^2)\eps^2+\mathcal O\big((\eps+|\sigma|)^3\big))>0$ for small enough $\sigma$ and positive $\eps$. Therefore, we conclude that $\max{\Re\lambda_i}>0$ if $\omega=\pm\frac{1}{2}$.\\
	 	Next, our goal is to factor out a cubic polynomial. For this purpose we will use the well-known decomposition procedure, that is, we make a substitution $\lambda=\mu-\frac{1}{3}a_2$. Then the general cubic equation becomes 
	 	\be
	 	\begin{split}
	 		\mu^3+3Q\mu-2R=0,
	 	\end{split}
	 	\ee
	 	where 
	 	\be
	 	\begin{split}
	 		Q&=\frac{3a_1-a_2^2}{9},\\
	 		R&=\frac{9a_2a_1-27a_0-2a_2^3}{54}.	
	 	\end{split}
	 	\ee
	 	Then one can factor out a cubic polynomial as
	 	\be
	 	\begin{split}
	 		\mu^3+3Q\mu-2R=(\mu-B)(\mu^2+B\mu+D+Q),
	 	\end{split}
	 	\ee
	 	where
	 	\be
	 	\begin{split}
	 		B&=\big[R+\sqrt{Q^3+R^2}\big]^{1/3}+\big[R-\sqrt{Q^3+R^2}\big]^{1/3},\\
	 		D&=\big[R+\sqrt{Q^3+R^2}\big]^{2/3}+\big[R-\sqrt{Q^3+R^2}\big]^{2/3}.	
	 	\end{split}
	 	\ee
	 	Therefore, the roots are
	 	\be\label{roots} 
	 		 		 	\begin{split}
	 		 		 	\lambda_1&=-\frac{1}{3}a_2+B,\\
	 		 		 		\lambda_{2,3}&=\frac{-\frac{2}{3}a_2-B\pm\sqrt{B^2-4(D+Q)}}{2}.
	 		 		 	\end{split}
	 		 		 	\ee
	 	
	  Next, we compute $Q$, $R$, $B$, $D$. First, we rewrite $a_i$ in the following fasion:
	 		 		 		 		 	\be
	 		 		 		 		 		 	\begin{split}
	 		 		 		 		 		 		a_0&=P_{04}\sigma^4+\mathcal O\big(\sigma^6(\eps+1)\big),\\
	 		 		 		 		 		 		a_1&=P_{12}\sigma^2+P_{14}\sigma^4+\mathcal O\big(|\sigma|^5(\eps+1)\big),\\
	 		 		 		 		 		 		a_2&=P_{20}+P_{22}\sigma^2+P_{24}\sigma^4+\mathcal O\big(|\sigma|^5(\eps+1)\big),
	 		 		 		 		 		 	\end{split}
	 		 		 		 		 		 	\ee
	 		 		 		 		 	where
	 		 		 		 		 		\be
	 		 		 		 		 		 	\begin{split}
	 		 		 		 		 		 P_{04}&=-4c(\eps)-64\omega^2\eps^2-288s^2A\eps^2+\mathcal O(\eps^3)+\sigma\mathcal O(\eps^2),\\
	 		 		 		 		 		 			P_{12}&=-5c(\eps)-72s^2A\eps^2-64\omega^2\eps^2+\mathcal O(\eps^3)+\sigma\mathcal O(\eps^2),\,\,P_{14}=24+\mathcal O(\eps),\\
	 		 		 		 		 		 		 	P_{20}&=-c(\eps)+\mathcal O(\eps^3)+\sigma\mathcal O(\eps^2),\,\,P_{22}=9+\mathcal O(\eps).
	 		 		 		 		 		 		 	\end{split}
	 		 		 		 		 		 		 	\ee
	 		 		 		 		 	\be
	 		 		 		 		 	\begin{split}
	 		 		 		 		 		Q&=\frac{3a_1-a_2^2}{9}=\frac{3P_{12}\sigma^2+3P_{14}\sigma^4-(P^2_{20}+2P_{20}P_{22}\sigma^2+2P_{20}P_{24}\sigma^4+P^2_{22}\sigma^4)+\mathcal O\big(|\sigma|^5(\eps+\sigma)\big)}{9}\\
	 		 		 		 		 		&=\frac{-P^2_{20}+(3P_{12}-2P_{20}P_{22})\sigma^2+(3P_{14}-2P_{20}P_{24}-P^2_{22})\sigma^4+\mathcal O\big(|\sigma|^5(\eps+1)\big)}{9},\\
	 		 		 		 		 		R&=\frac{9a_2a_1-27a_0-2a_2^3}{54}=\frac{9(P_{20}P_{12}\sigma^2+(P_{20}P_{14}+P_{22}P_{12})\sigma^4)-27P_{04}\sigma^4
	 		 		 		 		 			}{54}\\
	 		 		 		 		 			&\frac{-2(P^3_{20}+3P^2_{20}P_{22}\sigma^2+3(P_{20}P^2_{22}+P^2_{20}P_{24})\sigma^4)+\mathcal O(|\sigma|^5(\eps^3+\eps|\sigma|+|\sigma|))}{}\\
	 		 		 		 		 			&=\frac{-2P^3_{20}+(9P_{20}P_{12}-6P^2_{20}P_{22})\sigma^2+(9P_{20}P_{14}+9P_{22}P_{12}-27P_{04}-6P_{20}P^2_{22}-6P^2_{20}P_{24})\sigma^4
	 		 		 		 		 			}{54}\\
	 		 		 		 		 			&+\mathcal O(|\sigma|^5(\eps^3+\eps|\sigma|+|\sigma|)).	
	 		 		 		 		 	\end{split}
	 		 		 		 		 	\ee
	 		 		 		 		 	Next, we compute $Q^3+R^2$.
	 		 		 		 		 	\be
	 		 		 		 		 	\begin{split}
	 		 		 		 		 		Q^3+R^2&=\frac{-P^6_{20}+3P^4_{20}(3P_{12}-2P_{20}P_{22})\sigma^2-3P^2_{20}(3P_{12}-2P_{20}P_{22})^2\sigma^4}{3^6}\\
	 		 		 		 		 		&\frac{+3P^4_{20}(3P_{14}-2P_{20}P_{24}-P^2_{22})\sigma^4}{}+\frac{4P^6_{20}-4P^3_{20}(9P_{20}P_{12}-6P^2_{20}P_{22})\sigma^2}{4\cdot3^6}\\
	 		 		 		 		 			&\frac{-4P^3_{20}(9P_{20}P_{14}+9P_{22}P_{12}-27P_{04}-6P_{20}P^2_{22}-6P^2_{20}P_{24})\sigma^4
	 		 		 		 		 		}{}\\
	 		 		 		 		 		&\frac{+(9P_{20}P_{12}-6P^2_{20}P_{22})^2\sigma^4}{}+\mathcal O(|\sigma|^5(\eps+|\sigma|))=\big[-\frac{P^2_{20}P^2_{12}}{108}+\frac{P^3_{20}P_{04}}{27}\big]\sigma^4\\
	 		 		 		 		 		&+\mathcal O(\sum_{j=5}^{12}|\sigma|^j\eps^{12-j}).	
	 		 		 		 		 	\end{split}
	 		 		 		 		 	\ee
	 		 		 		 		 	
	 		 		 		 		 		Now, let $\sigma=\eps\hat\sigma$. Then,  the roots are
	 		 		 		 		 			\be \label{larges}
	 		 		 		 		 			\begin{split}
	 		 		 		 		 				\lambda_1&=(-\frac{1}{3}\tilde a_2+\tilde B)\eps^2,\\
	 		 		 		 		 				\lambda_{2,3}&=\frac{-\frac{2}{3}\tilde a_2-\tilde B\pm\sqrt{-3\tilde B^2-12\tilde Q}}{2}\eps^2.\\
	 		 		 		 		 			\end{split}
	 		 		 		 		 			\ee
	 		 		 		 		 			
	 		 		 		 		 		where
	 		 		 		 		 		\be\label{troots}
	 		 		 		 		 		\begin{split}
	 		 		 		 		 			-\frac{1}{3}\tilde a_2&=-\frac{2}{3}(1-4\omega^2)-3\hat{\sigma}^2+\mathcal O(\eps+|\sigma|)\sum_{j=0}^{1} C_{j}(\hat{\sigma}^2)^{j},\\
	 		 		 		 		 			\tilde B&=\big[-\frac{8}{27}(1-4\omega^2)^3+\hat\sigma^6+\mathcal O(\eps+|\sigma|)\sum_{j=0}^3 \tilde C_{j}(\hat{\sigma}^2)^{j}+\sum_{j=1}^2 \tilde{\tilde C}_{j}(\hat{\sigma}^2)^{j}
	 		 		 		 		 			+\sqrt{\tilde Q^3+\tilde R^2}\big]^{1/3}\\
	 		 		 		 		 			&+\big[-\frac{8}{27}(1-4\omega^2)^3+\hat\sigma^6+\mathcal O(\eps+|\sigma|)\sum_{j=0}^3 \tilde C_{j}(\hat{\sigma}^2)^{j}+\sum_{j=1}^2 \tilde{\tilde C}_{j}(\hat{\sigma}^2)^{j}-\sqrt{\tilde Q^3+\tilde R^2}\big]^{1/3},\\	 
	 		 		 		 		 			\tilde Q&=-\frac{4}{9}(1-4\omega^2)^2-\hat\sigma^4+\mathcal O(\eps+|\sigma|)\sum_{j=0}^2 {\hat C}_{j}(\hat{\sigma}^2)^{j}+\hat{\hat C}\hat\sigma^2	 	,\\
	 		 		 		 		 			\tilde R&=\frac{1}{\eps^6}R.	 		 
	 		 		 		 		 		\end{split}
	 		 		 		 		 		\ee
	 		 		 		 		 		Next, we fix $\omega$ such that $\omega^2\ne\frac{1}{4}$. Then we consider three different cases: 1) $|\hat \sigma|<<1$, 2) $1/C\leq|\hat \sigma|\leq C$, 3) $|\hat \sigma|>>1$.\\
	 		 		 		 		 		1) $|\hat \sigma|<<1$. We expand $\lambda_{1,2,3}$ w.r.t. $\hat \sigma$.
	 		 		 		 		 	Therefore, the expansions of $B$ and $D$ with respect to $\hat\sigma$ are of the form
	 		 		 		 		 	\be\label{R+-}
	 		 		 		 		 		 \begin{split}
	 		 		 		 		 		B&=[R+ \sqrt{Q^3+R^2}]^{1/3}+[R- \sqrt{Q^3+R^2}]^{1/3}=-\frac{ 2P_{20}}{3}+\frac{2}{3}R_1\frac{9}{P^2_{20}}\eps^2\hat\sigma^2+\big[\frac{2}{3}R_2\frac{9}{P^2_{20}}\\
	 		 		 		 		 		&+\frac{2}{9}(R_1^2+S_1^2)\frac{3^5}{P^5_{20}}\eps^6+12R_1\frac{(P'_{20})^2}{P^4_{20}}\big]\eps^4\hat\sigma^4+\mathcal O(\eps^2|\hat\sigma|^5(\eps+|\hat\sigma|)),\\
	 		 		 		 	 			D&=[R+ \sqrt{Q^3+R^2}]^{2/3}+[R- \sqrt{Q^3+R^2}]^{2/3}=\frac{ 2P^2_{20}}{9}-\frac{4}{3}R_1\frac{3}{P_{20}}\eps^2\hat\sigma^2+\big[-\frac{4}{3}R_2\frac{3}{P_{20}}\\
	 		 		 		 	 			&-\frac{2}{9}(R_1^2+S_1^2)\frac{3^4}{P^4_{20}}\eps^6-4R_1\frac{(P'_{20})^2}{P^3_{20}}\big]\eps^4\hat\sigma^4
	 		 		 			 		 		+\mathcal O(\eps^4|\hat\sigma|^5(\eps+|\hat\sigma|)),
	 		 		 		 		 		 	\end{split}
	 		 		 		 		 		 	\ee
	 		 		 		 		 		 	where
	 		 		 		 		 		 		 \be
	 		 		 		 		 		 		 \begin{split}
	 		 		 		 		 		 		 R_1&=\frac{9P_{20}P_{12}-6P^2_{20}P_{22}
	 		 		 		 		 		 		 		 		 			 			}{54},\\
	 		 		 		 	R_2&=\frac{9P_{20}P_{14}+9P_{22}P_{12}-27P_{04}-6P_{20}P^2_{22}-6P^2_{20}P_{24}
	 		 		 		 		 		 		 		 		 			 			}{54},\\
	 		 		 		S_1^2&=-\frac{P^2_{20}P^2_{12}}{108}+\frac{P^3_{20}P_{04}}{27},
	 		 		 		 		 		 		 		 		 	\end{split}
	 		 		 		 		 		 		 		 		 	\ee
	 		 		 		 		 		 		 	and $P'_{20}$ denotes the derivative of $P_{20}$ with respect to $\sigma$. Also, note that $P_{12}$ and $P_{20}$ are first degree polynomials with respect to $\sigma$.	 		 	
	 		 		 		 		 		 	Hence, using \eqref{roots} and \eqref{R+-}, we arrive at
	 		 		 		 		 		 		 	\be
	 		 		 		 		 		 		 	\begin{split}
	 		 		 		 		 		 		 	\lambda_1&=-P_{20}+[(-P_{22}+\frac{P_{12}}{P_{20}})]\eps^2\hat\sigma^2+\mathcal O(\eps^2|\hat\sigma|^3)\\
	 		 		 		 		 		 		 	&=-2(1-4\omega^2)\eps^2+\mathcal O(\eps^3)+\mathcal O(\eps^2)\sigma+(-\frac{36s^2}{27-2s^2}-\frac{4(1+4\omega^2)}{1-4\omega^2}+\mathcal O(\eps))\sigma^2+\mathcal O(\eps^2|\hat\sigma|^3).
	 		 		 		 		 		 		 	\end{split}
	 		 		 		 		 		 		 	\ee
	 		 		 		 		 		
	 		 		 		 		 	The other two roots are of the form
	 		 		 		 			\be \label{quad1}
	 		 		 		 		 		 	\begin{split}
	 		 		 		 		 		 		\lambda_{2,3}&=\frac{(-5+\frac{36s^2}{27-2s^2}+\frac{32\omega^2}{1-4\omega^2})\sigma^2+\mathcal O(\eps^2|\hat\sigma|^3)\pm\sqrt{(5-\frac{36s^2}{27-2s^2}-\frac{32\omega^2}{1-4\omega^2})^2\sigma^4}}{2}\\
	 		 		 		 		 		 				&\frac{\overline{-4(4-\frac{144s^2}{27-2s^2}-\frac{32\omega^2}{1-4\omega^2})\sigma^4+\mathcal O(\eps^4|\hat\sigma|^5(\eps+|\hat\sigma|))}}{}
	 		 		 		 		 		 	\end{split}
	 		 		 		 		 		 	\ee
	 		 		 		 		 		 	If $4-\frac{144s^2}{27-2s^2}-\frac{32\omega^2}{1-4\omega^2}<0$, then $\max{\Re\lambda_i}>0$. And if $4-\frac{144s^2}{27-2s^2}-\frac{32\omega^2}{1-4\omega^2}>0$, then $\max{\Re\lambda_i}\leq0$.\\
	 		 		 		 		 	 Finally,
	 		 		 		 		 		 		 	\be\label{eBM}
	 		 		 		 			\begin{split}
	 		 		 		 		 		&\lambda_{2}(\eps,\omega,s,\sigma)=(\lambda_-+\mathcal O(\eps))\eps^2\hat\sigma^2+\mathcal O(\eps^2|\hat\sigma|^3),\\
	 		 		 		 		 	&\lambda_{3}(\eps,\omega,s,\sigma)=(\lambda_++\mathcal O(\eps))\eps^2\hat\sigma^2+\mathcal O(\eps^2|\hat\sigma|^3),
	 		 		 		 		 		 		 		\end{split}
	 		 		 		 		 		 		 			\ee
	 		 		 		 		 		 		 			where
	 		 		 		 		 		 		 			\be 
	 		 		 		 		 		 		 			\begin{split}
	 		 		 		 		 		 		 				\lambda_{\pm}&=\frac{(-5+\frac{36s^2}{27-2s^2}+\frac{32\omega^2}{1-4\omega^2})\pm\sqrt{(5-\frac{36s^2}{27-2s^2}-\frac{32\omega^2}{1-4\omega^2})^2-4(4-\frac{144s^2}{27-2s^2}-\frac{32\omega^2}{1-4\omega^2})}}{2}.\\	
	 		 		 		 		 		 		 			\end{split}
	 		 		 		 		 		 		 			\ee
	 		 		 		 	
	 		 		 		 	2) $1/C\leq|\hat \sigma|\leq C$.\\
	 							For the unperturbed case, we know (see Section 
	 							\ref{s:GLcompare} for details) that $\Re\lambda_j\leq\eta<0$ if $4-\frac{144s^2}{27-2s^2}-\frac{32\omega^2}{1-4\omega^2}>0$. Since $\hat\sigma$ belongs to the compact interval, we deduce that $\Re\lambda_j\leq\hat\eta<0$.\\
	 		 		 		 	3) $|\hat \sigma|>>1$. It follows from formulas \eqref{larges} and \eqref{troots} that the roots $\lambda_1$, $\lambda_2$ and $\lambda_3$ are controlled by $-\sigma^2$, $-4\sigma^2$ and $-4\sigma^2$ respectively.

\end{proof}

\section{Comparison with modified Ginzburg--Landau approximation}\label{s:GLcompare}
In this section we study in more detail the various operations in the modified Ginzburg--Landau expansion, 
showing that, after natural preconditioning passes on each side, these can be matched step by step
with those of the exact Lyapunov--Schmidt reduction procedure.
This gives a deeper explanation why the two procedures give the same expansion to their common order of approximation.
In the process, we carry out the proof of Theorem \ref{dispthm}.

\subsection{Modified Ginzburg--Landau derivation through multiscale expansion}\label{s:GLexp}
Following \cite{MC00}, we start by deriving in detail the modified Ginzburg--Landau system.
The derivation is based on the ansatz
\be 
\begin{split}
	& u(t,x)\approx U_A(\hat t,\hat x)=\frac{1}{2}\eps A(\hat t,\hat x)e^{ix}+c.c.+\eps^2B(\hat t,\hat x)+\frac{1}{2}\eps^2e^{2ix}C(\hat t,\hat x)+c.c.+h.o.t.,
\end{split}
\ee
where $(\hat t,\hat x)=(\eps^2 t,\eps x)$.

Substituting this ansatz into \eqref{sh1c} and collecting terms of the form $\eps^{j_1}e^{i\frac{1}{2}{j_2}x}$, we arrive at the equations:
\be 
\begin{split}\label{steps}
	\eps e^{ix} :\quad 0=	& -\frac{1}{2}A\partial_x^2\big[ -(1+\partial_x^2)^2  \big]e^{ix},\\
	\eps^2e^{0ix} :\quad 0=	&\frac{1}{2}s|A|^2\partial_x^2(e^{0ix}),\\
	\eps^2 e^{ix} :\quad 0=	& -\frac{1}{4}\partial_{\hat x}A\big(4\partial_x\big[ -(1+\partial_x^2)^2  \big]e^{ix}+8\partial_x\big[ -(1+\partial_x^2)\partial_x^2  \big]e^{ix}\big),\\
	\eps^2e^{2ix} :\quad 0=&\frac{1}{2}(-36 Ce^{2ix}-2s A^2e^{2ix}),\\
	\eps^3e^{ix} :\quad \partial_{\hat t} A e^{ix}=& A e^{ix}+4\partial^2_{\hat x} A  e^{ix}+(-2sAB+\frac{s^2}{18}|A|^2A) e^{ix}-\frac{3}{4}|A|^2A e^{ix},\\
	\eps^4e^{0ix} :\quad \partial_{\hat t} B=&\partial^2_{\hat x} B+ \frac{1}{2} s\partial^2_{\hat x}(|A|^2).
\end{split}
\ee

Hence, we arrive at the modified Ginzburg--Landau system:
\be 
\begin{split}\label{GL1}
	\partial_{\hat t} A =& 4\partial^2_{\hat x} A+A  -\frac{27-2s^2}{36}|A|^2A-2sAB,\\
	\partial_{\hat t} B=&\partial^2_{\hat x} B+ \frac{1}{2} s\partial^2_{\hat x}(|A|^2).
\end{split}
\ee

Note that \eqref{GL1} has the explicit solution:
\be 
\begin{split}\label{GL2}
	A_{\omega,s}(\hat x)=6\sqrt{\frac{1-4\omega^2}{27-2s^2}}e^{i\omega\hat x},\,\,B(\hat x)=0,\,\,\omega\in[-\frac{1}{2},\frac{1}{2}],\,\,s\in (-\sqrt{27/2},\sqrt{27/2}).
\end{split}
\ee
Hence, if we go through the steps in formulas \eqref{steps}-\eqref{GL2}, replace $\partial_{\hat x}$ by $i\omega$, $B(\hat x)$ by $0$ and ignore $\hat t $ dependence of $A$, we will arrive at the following equation for $A$:
$$
-4\omega^2A+A-\frac{27-2s^2}{36}|A|^2A=0.
$$

\begin{remark}
	Note that all coefficients in the first equation of \eqref{GL1} can be been set to unity by rescaling $\hat t,\hat x, A$ and $B$.
\end{remark}

\subsection{Existence: exact theory vs. modified Ginzburg--Landau approximation}
Now that we know the precise scaling in the existence part we use the following ansatz to go through the existence steps and compare them to the steps of the modified Ginzburg--Landau derivation:
$$
 u=\alpha\cos \xi\eps +V(\eps,\alpha).
$$
We substitute this ansatz into the equation
\be 
\begin{split}\label{exist}
& (I-Q)N(\eps,k(\eps, \omega), u)=0.
\end{split}
\ee

1.) First of all, it is clear that $V(0,\alpha)=0$.\\
2.) Now, we differentiate  equation \eqref{exist} with respect to $\eps$.
\be\label{eps}
\begin{split}
	& \partial_{\eps}(I-Q)N(\eps,k, \alpha\eps U_1+ V)=(I-Q)\big\{\big(-2\omega\partial_\xi^2[ -(1+k^2\partial_\xi^2)^2+\eps^2]-k^2\partial_\xi^2[ -4\omega(1+k^2\partial_\xi^2)\partial_\xi^2+2\eps]\big) u\\
	&-k^2\partial_\xi^2\big( -(1+k^2\partial_\xi^2)^2+\eps^2-(2s u+ 3u^2)\big)(\alpha U_1+\partial_{\eps}V)  \big\}=0,
\end{split}
\ee
where $u=\alpha\eps U_1 + V$.\\
Hence, by step 1.),
\be
\begin{split}
 (I-Q)L_{per}\partial_{\eps}V|_{\eps=0}=-\alpha(I-Q)L_{per}U_1=0.
 \end{split}
 \ee
 Since $(I-Q)L_{per}(I-P)$ is invertible, $\partial_{\eps}V|_{\eps=0}=0$.

3.) Next, we would like to compute $\partial^2_{\eps}V|_{\eps=0}$.  Differentiating \eqref{eps} with respect to $\eps$, we obtain
%
\be\label{eps2}
\begin{split}
	& \partial^2_{\eps}(I-Q)N(\eps,k, \alpha\eps U_1+ V)=(I-Q)\big\{\big(-4\omega\partial_\xi^2[ -4\omega(1+k^2\partial_\xi^2)\partial_\xi^2+2\eps]-k^2\partial_\xi^2[ -8\omega^2\partial_\xi^4+2]\big) u\\
		&\big(-2\omega\partial_\xi^2[ -2(1+k^2\partial_\xi^2)^2+2\eps^2-(2s u+ 3u^2)]-k^2\partial_\xi^2[ -8\omega(1+k^2\partial_\xi^2)\partial_\xi^2+4\eps\\
		&-(2s + 6u)(\alpha U_1+\partial_{\eps}V)]\big)(\alpha U_1+\partial_{\eps}V)-k^2\partial_\xi^2\big( -(1+k^2\partial_\xi^2)^2+\eps^2-(2s u+ 3u^2)\big)\partial^2_{\eps}V  \big\}=0,
\end{split}
\ee
Or,
\be
\begin{split}
	& \partial^2_{\eps}(I-Q)N(\eps,k, \alpha\eps U_1+ V)_{\eps=0}=(I-Q)\big\{
		\big(-2\omega\partial_\xi^2[ -2(1+\partial_\xi^2)^2]-\partial_\xi^2[ -8\omega(1+\partial_\xi^2)\partial_\xi^2\\
		&-2s\alpha U_1]\big)(\alpha U_1)-\partial_\xi^2\big( -(1+\partial_\xi^2)^2\big)\partial^2_{\eps}V  \big\}=0,
\end{split}
\ee
Note that 
\be
\begin{split}
    & s\alpha^2(I-Q)\partial_\xi^2(1)=0,\\
	& -\alpha(I-Q)
		\big(4\omega\partial_\xi^2[ -(1+\partial_\xi^2)^2]+8\omega\partial_\xi^2[ -(1+\partial_\xi^2)\partial_\xi^2
		]\big) U_1=0.
\end{split}
\ee
Therefore, we obtain
\ba \label{formula1^2}
	(I-Q)L_{per}\partial^2_{\eps}V|_{\eps=0}& =4s\alpha^2\cos 2\xi.
\ea
Since $span\{\cos 2\xi\}$ is an invariant subspace for the invertible operator $(I-Q)L_{per}(I-P)$, $\frac{1}{2}\partial^2_{\eps}V|_{\eps=0}$ is of the form $C\cos 2\xi$ and have the following equation for $C$:
$$
\begin{aligned}
	-36C\cos 2\xi& =2s\alpha^2\cos 2\xi.
\end{aligned}
$$

So far, we have shown that 
\be
\begin{split}\label{uu}
	&u=\alpha\cos \xi\eps -\frac{1}{18}s\alpha^2\cos 2\xi\eps^2+\mathcal O(\eps^3).
\end{split}
\ee
Also, notice that in formula \eqref{uu} $\mathcal O(\eps^3)=\mathcal O(\a\eps^3)$. Then
\be
\begin{split}\label{u}
	&u=\alpha\cos \xi\eps -\frac{1}{18}s\alpha^2\cos 2\xi\eps^2+\mathcal O(\a\eps^3).
\end{split}
\ee

In order to obtain the reduced system, we plug \eqref{u} into the first equation from \eqref{projection equations}, obtaining
$$
\begin{aligned}
& \tilde QN(\eps,k(\eps), u)=-k^2\tilde Q\partial_\xi^2\big[ -(1+k^2\partial_\xi^2)^2 u+\eps^2 u-s u^2- u^3  \big] =0.
\end{aligned}
$$

Hence, we have the reduced system:
\be
\begin{split}
-4\omega^2\alpha\eps^3+\a\eps^3 -\frac{27-2s^2}{36}\a^3\eps^3+\mathcal O(\a\eps^4)=0.
\end{split}
\ee
It is equivalent to 
\be
\begin{split}
	-4\omega^2\alpha+\a -\frac{27-2s^2}{36}\a^3+\mathcal O(\a\eps)=0.
\end{split}
\ee

Note that, under the imposed scaling, the computations of the reduced (existence) equation derived by Lyapunov--Schmidt
reduction agree at each order with derived by the formal modified Ginzburg--Landau approximation.
\subsection{Stability: exact vs. modified Ginzburg--Landau linearized dispersion relations}\label{s:ccred}
Next, we derive the linearized dispersion relations for the modified Ginzburg--Landau system:
\be 
\begin{split}
	\partial_{\hat t} A =& 4\partial^2_{\hat x} A+A  -\frac{27-2s^2}{36}|A|^2A-2sAB,\\
	\partial_{\hat t} B=&\partial^2_{\hat x} B+ \frac{1}{2} s\partial^2_{\hat x}(|A|^2).
\end{split}
\ee

In order to study the linearized stability of 	$A_{\omega,s}(\hat x)=6\sqrt{\frac{1-4\omega^2}{27-2s^2}}e^{i\omega\hat x},\,\,B(\hat x)=0,\,\,\omega\in[-\frac{1}{2},\frac{1}{2}],\,\,s\in (-\sqrt{27/2},\sqrt{27/2})$, we would like to derive the linearized equation for the model \eqref{sh1c} around $A_{\omega,s}(\hat x)$ using the equation 

\be
\begin{split}\label{B1}
\partial_{ t}B
	& =-\partial_{ x}^2\big[ -(1+\partial_{ x}^2)^2+\eps^2-2s\tilde u-3\tilde u^2) \big]B,
\end{split}
\ee
where
\be 
\begin{split}
&\tilde u=U_A(\hat t,\hat x)=\frac{1}{2}\eps A_{\omega,s}(\hat x)e^{ix}+c.c.+\frac{1}{2}\eps^2C(\hat x)e^{2ix}+c.c.+h.o.t.,\\
&C(\hat x)=-\frac{sA^2_{\omega,s}(\hat x)}{18},
\end{split}
\ee
and the ansatz
\be 
\begin{split}
	&\B(\hat t,\hat x)=b_re^{i(\omega\hat x+x)}-ib_ie^{i(\omega\hat x+x)}+c.c.+\eps b_z+\eps e^{i(\omega\hat x+2x)}(\Phi_2(\hat t,\hat x)-i\tilde\Phi_2(\hat t,\hat x)+c.c.)+h.o.t.
\end{split}
\ee
	
Substituting this ansatz into \eqref{B} and collecting terms of the form $(-i)^{j_1}\eps^{j_2}e^{i(\omega\hat x+x)}$, we arrive at the equations:

\be 
\begin{split}
	e^{i(\omega\hat x+x)}:\quad 0=	& -b_re^{i\omega\hat x}\partial_{ x}^2\big[ -(1+\partial_{ x}^2)^2]e^{ix},\\
	-ie^{i(\omega\hat x+x)}:\quad 0=	& -b_ie^{i\omega\hat x}\partial_{ x}^2\big[ -(1+\partial_{ x}^2)^2]e^{ix},\\
	\eps e^{i(\omega\hat x+2x)} :\quad 0
		=&-36\Phi_2-24s\sqrt{\frac{1-4\omega^2}{27-2s^2}}b_r,\\
	-i\eps e^{i(\omega\hat x+2x)} :\quad 0
			=&-36\tilde\Phi_2-24s\sqrt{\frac{1-4\omega^2}{27-2s^2}}b_i.\\
\end{split}
\ee

Therefore, $\Phi_2=-\frac{2s}{3}\sqrt{\frac{1-4\omega^2}{27-2s^2}}b_r$ and $\tilde\Phi_2=-\frac{2s}{3}\sqrt{\frac{1-4\omega^2}{27-2s^2}}b_i$.\\
We also arrive at the compatibility conditions:
\be 
\begin{split}
	\eps^2 e^{i(\omega\hat x+x)}:\quad& \partial_{\hat t}b_r=	4\partial^2_{\hat x}b_r-4
		\omega^2b_r+8\omega\hat
		\partial_{\hat x}b_i+(1-\frac{54(1-4\omega^2)}{27-2s^2}-(1-4\omega^2)+\frac{4s^2(1-4\omega^2)}{27-2s^2})b_r\\
		&-12s\sqrt{\frac{1-4\omega^2}{27-2s^2}}b_z,\\
	-i\eps^2 e^{i(\omega\hat x+x)}:\quad& \partial_{\hat t}b_i=-8\omega\hat
				\partial_{\hat x}b_r+	4\partial^2_{\hat x}b_i-4
			\omega^2b_i+(1-\frac{54(1-4\omega^2)}{27-2s^2}+(1-4\omega^2)+\frac{4s^2(1-4\omega^2)}{27-2s^2})b_i,\\
			\eps^3:\quad& \partial_{\hat t}b_z=6s\sqrt{\frac{1-4\omega^2}{27-2s^2}}\partial^2_{\hat x}b_r+\partial^2_{\hat x}b_z.
\end{split}
\ee

Hence, the spectral matrix of the linearized operator is of the form
\be\label{gldisp}
\begin{split}
	0
	&=\bp -4\hat\sigma^2-2(1-4\omega^2)-\hat\lambda & 8i\omega\hat\sigma &-12s\sqrt{\frac{1-4\omega^2}{27-2s^2}}\\ -8i\omega\hat\sigma & -4\hat\sigma^2-\hat\lambda&0\\
	-6s\sqrt{\frac{1-4\omega^2}{27-2s^2}}\hat\sigma^2&0&-\hat\sigma^2-\hat\lambda \ep\bp \b_1 \\ \b_2\\\b_3 \ep.
\end{split}
\ee

Then, for $|\hat\sigma|<<1$
	\be\label{eGL}
	 		 		 \begin{split}
	 		 		 	&\hat\lambda_{1}(\hat\sigma)=-2(1-4\omega^2)+(-\frac{36s^2}{27-2s^2}-\frac{4(1+4\omega^2)}{1-4\omega^2})\hat\sigma^2+\mathcal O(|\hat\sigma|^3),\\
	 		 		 	&\hat\lambda_{2}(
	 		 		 	\hat\sigma)=\lambda_-\hat\sigma^2
	 		 		 	+\mathcal O(|\hat\sigma|^3),\\
	 		 		 	&\hat\lambda_{3}(
	 		 		 	\hat\sigma)=\lambda_+\hat\sigma^2
	 		 		 	+\mathcal O(|\hat\sigma|^3),\\
	 		 		 	&\lambda_{\pm}=\frac{(-5+\frac{36s^2}{27-2s^2}+\frac{32\omega^2}{1-4\omega^2})\pm\sqrt{(5-\frac{36s^2}{27-2s^2}-\frac{32\omega^2}{1-4\omega^2})^2-4(4-\frac{144s^2}{27-2s^2}-\frac{32\omega^2}{1-4\omega^2})}}{2}.
	 		 		 \end{split}
 		 		 			\ee

Note that, as in the existence part, the derivation by rigorous Lyapunov--Schmidt reduction, in the Ginzburg--Landau scaling, 
agrees at each step/order with that by formal modified Ginzburg--Landau approximation as can be seen below.

Now that we know the explicit form of solution in the existence part we use the scaling $\lambda=\eps^2\hat\lambda, \,\sigma=\eps\hat \sigma$ to go through the spectral matrix steps and compare them to the steps in section 5.\\
We now consider the eigenvalue problem of $B(\eps, \omega, s,\sigma)$:
\be
0=\Big[ B(\eps, \omega, s,\hat\sigma\eps)-\hat\l\eps^2 I \Big] W.
\ee

Now, in order to use Lyapunov--Schmidt reduction, we first define the zero eigenprojection
\be \label{projection P}
\hat Qu= \langle  U_1, u\rangle U_1+\langle  U_2, u\rangle U_2+\frac{1}{2\eps^2}\langle  \eps U_3, u\rangle \eps U_3, \quad \text{where} \quad \langle u, v \rangle = \frac{1}{\pi} \int_0^{2\pi}u \cdot v d\xi,
\ee
and define the mapping
\be \label{projection tilde P}
\tilde { \hat Q}: L^2_{per}([0,2\pi], \RR) \rightarrow \RR^3 ; u \mapsto (\langle U_1, u\rangle, \langle U_2, u\rangle, \frac{1}{2\eps^2}\langle \eps U_3, u\rangle)^{T}
\ee

Next, we decompose $W = \b_1 U_{1}+\b_{2} U_{2}+\b_3\eps U_{3} + \mathcal V$ and we first solve
\be\label{scaling}
\begin{split}
	0
	& =(I-Q)\Big[ B(\eps, \omega, s,\hat\sigma\eps)-\hat\l\eps^2 I\Big] (\b_1 U_{1}+\b_{2} U_{2}+\b_3\eps U_{3}  + \mathcal V),
\end{split}
\ee
It is clear that the relation between $\b$ and $\mathcal V$ is linear. Then, let $\mathcal V=\mathcal V_1\b_1+\mathcal V_2\b_2+\mathcal V_3\b_3$. Now let us find asymptotic expansions of $\mathcal V_1$, $\mathcal V_2$ and $\mathcal V_3$ with respect to parameter $\eps$.\\
1.) First, we compute $\mathcal V_i|_{\eps=0}$. \\
\be
\begin{split}
	0
	& =(I-\hat Q) B(0,\omega,s,0) (\b_1 (U_{1}+\mathcal V_1|_{\eps=0})+\b_{2} (U_{2}+\mathcal V_2|_{\eps=0})+\b_{3} \mathcal V_3|_{\eps=0}).
\end{split}
\ee

Since $(I-\hat Q) B(0,\omega,s,0) (\b_1 U_{1}+\b_{2} U_{2})=0$ and $(I-\hat Q) B(0,\omega,s,0)(I-P)$ is invertible, $\mathcal V_i|_{\eps=0}=0$.

2.) Now, we differentiate  equation \eqref{scaling} with respect to $\eps$ and plug in $0$ for $\eps$.
\be
\begin{split}
	0
	& =(I-\hat Q)\big\{\partial_{\eps}B(0,\omega,s,0) (\b_1 U_{1}+\b_{2} U_{2})+ B(0,\omega,s,0)(\b_1 \partial_{\eps}\mathcal V_1|_{\eps=0}+\b_{2} \partial_{\eps}\mathcal V_2|_{\eps=0}+\b_{3} \partial_{\eps}\mathcal V_3|_{\eps=0})\big\}.
\end{split}
\ee
Or,
\be\label{derPhi}
\begin{split}
	0
	& =(I-\hat Q)\big\{-\partial_{ \xi}^2\{-12s\sqrt{\frac{1-4\omega^2}{27-2s^2}}\cos\xi(\b_1 U_{1}+\b_{2} U_{2})\}+ L_{per}(\b_1 \partial_{\eps}\mathcal V_1|_{\eps=0}+\b_{2} \partial_{\eps}\mathcal V_2|_{\eps=0}+\b_{3} \partial_{\eps}\mathcal V_3|_{\eps=0})\big\}\\
	&=-24s\sqrt{\frac{1-4\omega^2}{27-2s^2}}(\b_1\cos2\xi+\b_2\sin2\xi)+(I-\hat Q)L_{per}(\b_1 \partial_{\eps}\mathcal V_1|_{\eps=0}+\b_{2} \partial_{\eps}\mathcal V_2|_{\eps=0}+\b_{3} \partial_{\eps}\mathcal V_3|_{\eps=0}).
\end{split}
\ee
Therefore,

\be\label{Phi}
\begin{split}
	&\b_1 \partial_{\eps}\mathcal V_1|_{\eps=0}=\Phi_2\cos2\xi,\\
	&\b_{2} \partial_{\eps}\mathcal V_2|_{\eps=0}= \tilde\Phi_2 \sin2\xi,\\
	&\b_{3} \partial_{\eps}\mathcal V_3|_{\eps=0}=0.
\end{split}
\ee

Substituting \eqref{Phi} into \eqref{derPhi} leads to
\be
\begin{split}
	0
	&=-24s\sqrt{\frac{1-4\omega^2}{27-2s^2}}\b_1-36\Phi_2,\\
	0
		&=-24s\sqrt{\frac{1-4\omega^2}{27-2s^2}}\b_2-36\tilde\Phi_2.
\end{split}
\ee

Hence, $W = \b_1 U_{1}+\b_{2} U_{2}+\b_{3}\eps U_{3} + \mathcal V=(U_1-\frac{2s}{3}\sqrt{\frac{1-4\omega^2}{27-2s^2}}\cos2\xi\eps+\mathcal O(\eps^2) )\b_{1}+ (U_{2}-\frac{2s}{3}\sqrt{\frac{1-4\omega^2}{27-2s^2}} \sin2\xi\eps+\mathcal O(\eps^2) )\b_{2}+(\eps U_{3} + \mathcal O(\eps^2))\b_{3}$. Next, we plug $W$ into 
\be
\begin{split}
	0
	& =\tilde {\hat Q}\Big[ B(\eps, \omega, s,\hat\sigma\eps)-\hat\l\eps^2 I\Big] W.
\end{split}
\ee
Then we arrive at
\be 
\begin{split}
	0=	& -\hat\lambda\eps^2\b_1-(4\hat\sigma^2+4
	\omega^2)\eps^2\b_1+8i\omega\hat\sigma
	\eps^2\b_2+(1-\frac{54(1-4\omega^2)}{27-2s^2}-(1-4\omega^2)+\frac{4s^2(1-4\omega^2)}{27-2s^2})\eps^2\b_1\\
	&-12s\sqrt{\frac{1-4\omega^2}{27-2s^2}}\eps^2\b_3,\\
	0=	&-8i\omega\hat\sigma
	\eps^2\b_1 -\hat\lambda\eps^2\b_2-(4\hat\sigma^2+4
	\omega^2)\eps^2\b_2+(1-\frac{54(1-4\omega^2)}{27-2s^2}+(1-4\omega^2)+\frac{4s^2(1-4\omega^2)}{27-2s^2})\eps^2\b_2,\\
	0=&-6s\sqrt{\frac{1-4\omega^2}{27-2s^2}}\hat \sigma^2\eps^2\b_1-\hat\lambda\eps^2\b_3-\hat\sigma\eps^2\b_3.
\end{split}
\ee
Or,
\be
\begin{split}
	0
	&=\bp -4\hat\sigma^2-2(1-4\omega^2)-\hat\lambda & 8i\omega\hat\sigma &-12s\sqrt{\frac{1-4\omega^2}{27-2s^2}}\\ -8i\omega\hat\sigma & -4\hat\sigma^2-\hat\lambda&0\\
	-6s\sqrt{\frac{1-4\omega^2}{27-2s^2}}\hat\sigma^2&0&-\hat\sigma^2-\hat\lambda \ep\bp \b_1 \\ \b_2\\\b_3 \ep\eps^2+O(\eps^3).
\end{split}
\ee

This yields in passing agreement of the critical linearized dispersion relations.

\begin{proof}[Proof of Theorem \ref{dispthm}]
	Note that the exact reduced spectral system given by \eqref{spmatrixs} and Corollary \ref{errcorr}
	agrees after Ginzburg Landau scaling ($\sigma =:\eps \hat \sigma$, $\lambda_j=:\eps^2 \hat \lambda_j$) to appropriate order with the matrix eigenvalue problem \eqref{gldisp}.
	Also, it follows from formulas \eqref{eBM} and \eqref{eGL} that the roots likewise agree to lowest order, 
	giving the result, \eqref{exactdisp}.
	
\end{proof}

{\bf Acknowledgments.}
We are grateful to Prof. K. Zumbrun for many suggestions and helpful discussions.


\end{document}